\documentclass[11pt]{article}
\usepackage{color}
\usepackage{amsfonts}
\usepackage{graphicx}
\usepackage{amsmath,enumerate}
\usepackage{cases}
\usepackage{amsthm}
\usepackage{amssymb}
\usepackage{amsfonts}
\usepackage{latexsym,amssymb,amsfonts}
\usepackage{mathrsfs}
\usepackage{graphicx}
\usepackage{psfrag}
\usepackage{amsmath, amsbsy}
\usepackage{amsopn, amstext}
\usepackage{amsmath,multicol,amsopn}
\usepackage{latexsym,amssymb}
\usepackage{hyperref}

\usepackage{lineno}
\def\dbE{{\mathbb{E}}}
\def\dbF{{\mathbb{F}}}

\def\dbH{{\mathbb{H}}}
\def\dbI{{\mathbb{I}}}

\def\dbP{{\mathbb{P}}}

\def\dbR{{\mathbb{R}}}
\def\dbS{{\mathbb{S}}}

\def\d{\delta}
\def\e{\varepsilon}

\def\k{\kappa}

\def\t{\tau}
\def\f{\varphi}

\def\Th{\Theta}

\def\3n{\negthinspace \negthinspace \negthinspace }
\def\2n{\negthinspace \negthinspace }
\def\1n{\negthinspace }
\def\ns{\noalign{\smallskip} }

\def\ds{\displaystyle}
\def\G{\Gamma}
\def\D{\Delta}
\def\Th{\Theta}

\def\Om{\Omega}

\def\cA{{\cal A}}

\def\cC{{\cal C}}

\def\cF{{\cal F}}
\def\cG{{\cal G}}

\def\cP{{\cal P}}
\def\cQ{{\cal Q}}
\def\cR{{\cal R}}

\def\cU{{\cal U}}

\def\mE{{\mathbb{E}}}

\def\ms{\medskip}

\def\q{\quad}
\def\qq{\qquad}
\def\hb{\hbox}

\def\liminf{\mathop{\underline{\rm lim}}}

\def\esssup{\mathop{\rm esssup}}

\def\wt{\widetilde}
\def\cd{\cdot}

\def\deq{\mathop{\buildrel\D\over=}}

\def\({\Big (}
\def\){\Big )}
\def\[{\Big[}
\def\]{\Big]}

\def\={\buildrel \triangle \over =}

\def\ee{\end{equation}}
\def\bea{\begin{eqnarray}}
\def\eea{\end{eqnarray}}
\def\bt{\begin{theorem}}
\def\et{\end{theorem}}
\def\bc{\begin{corollary}}
\def\ec{\end{corollary}}
\def\bl{\begin{lemma}}
\def\el{\end{lemma}}
\def\bp{\begin{proposition}}
\def\ep{\end{proposition}}
\def\br{\begin{remark}}
\def\er{\end{remark}}
\def\ba{\begin{array}}
\def\ea{\end{array}}
\def\bde{\begin{definition}}
\def\ede{\end{definition}}

\newtheorem{lemma}{Lemma}[section]
\newtheorem{remark}{Remark}[section]

\newtheorem{theorem}{Theorem}[section]
\newtheorem{corollary}{Corollary}[section]

\newtheorem{definition}{Definition}[section]
\newtheorem{proposition}{Proposition}[section]

\def\punct{}
\newtheoremstyle{dotless}{}{}{\rm}{}{\bf}{\punct}{.5em}{}
\theoremstyle{dotless}

\newtheorem{taggedassumptionx}{}
\newenvironment{taggedassumption}[1]
 {\renewcommand\thetaggedassumptionx{#1}\taggedassumptionx}
 {\endtaggedassumptionx}

\allowdisplaybreaks[4]
\makeatletter

\@addtoreset{equation}{section}
\makeatother

\title{\bf
Solvability of Equilibrium  Riccati Equations:
A Direct Approach\footnote{Although we  provide a new approach to the ERE associated time-inconsistent stochastic LQ controls, the main result of this manuscript has been obtained by Yong \cite{Yong-2017}. Thus, we just  maintain it  as a preprint for the interested readers, and will not publish it officially.}
}

\author{Bowen Ma\thanks{College of Mathematics and Physics, Chengdu University of Technology, Chengdu, 610059, China
(Email: {\tt albertmabowen@gmail.com}). }
~~and~~
Hanxiao Wang\thanks{Corresponding author. School of Mathematical Sciences, Shenzhen University, Shenzhen,
518060, China (Email: {\tt hxwang@szu.edu.cn}). This author is supported in part by NSFC Grant 12201424,
Guangdong Basic and Applied Basic Research Foundation 2023A1515012104,
and the Science and Technology Program of Shenzhen RCBS20231211090537064.
During preparing this work, my daughter was born, whose pet name is ``$\pi\pi$".
Many thanks for my wife's support and understanding.
}
}

\date{\today}

\begin{document}
\maketitle

\noindent \textbf{Abstract.}
The solvability of  equilibrium  Riccati equations (EREs) plays a central role in the study of
time-inconsistent stochastic linear-quadratic optimal control problems,
because it paves the way to constructing a closed-loop equilibrium strategy.
Under the standard conditions, Yong \cite{Yong-2017}  established its well-posedness
by introducing the well-known multi-person differential game method.
However, this method  depends on the  dynamic programming principle (DPP) of
the sophisticated problems on every subinterval, and thus is essentially a control theory approach.
In this paper, we shall give a new and more direct proof, in which the DPP is no longer needed.
We first establish a priori estimates for the ERE in the case of smooth coefficients.
Using this estimate, we then demonstrate both the local and global solvability of the ERE
by constructing an appropriate Picard iteration sequence,
which actually provides a numerical algorithm.
Additionally, a mollification method is employed to handle the case with non-smooth coefficients.



 \medskip

\section{Introduction}\label{sec-intro}

Time-inconsistent stochastic optimal control theory has attracted  strong attentions
in the last two decades for its wide application, especially in behavioral finance.
An important special case is the so-called  time-inconsistent stochastic linear-quadratic (SLQ, for short) optimal control
problem, whose state equation and cost functional can be described by
\begin{equation}\label{state}
\left\{\begin{aligned}
&dX(s)=[A(s)X(s)+B(s)u(s)]ds  \\
&\qq\qq +[C(s)X(s)+D(s)u(s)]dW(s), \q s\in [t,T],\\
&X(t)=x,
\end{aligned}\right.
\end{equation}
and
\begin{align}
J(t,x;u(\cd))=&\mE_t\Big[ \int_{t}^{T} \big(\langle Q(t,s)X(s),X(s) \rangle +\langle R(t,s)u(s),u(s) \rangle \big)ds
+  \langle G(t)X(T),X(T) \rangle  \Big],\label{cost}
\end{align}
respectively.
In the above, $T>0$ denotes a fixed time horizon,  $W(\cd)$ represents a one-dimensional standard Brownian motion defined on the
complete filtered probability space $(\Omega,\cF,\dbF,\dbP)$ with $\dbF=\{\cF_t\}_{t\geq 0}$ being the augmented natural filtration generated by $W(\cd)$,
and $\dbE_t[\,\cd\,]=\dbE[\,\cd\,|\cF_t]$ is the conditional expectation operator.
Moreover, $A(\cd),C(\cd):[0,T]\to\dbR^{n\times n}$ and $B(\cd),D(\cd):[0,T]\to\dbR^{n\times m}$
are given coefficients, and  $Q(\cd,\cd):[0,T]^2\to\dbS^{n}$, $R(\cd,\cd):[0,T]^2\to\dbS^{m}$, and $\G(\cd):[0,T]\to\dbS^{n}$
are the weighting matrices. For any given {\it initial pair} $(t,x)\in[0,T)\times L^2_{\cF_t}(\Omega;\,\dbR^n)$,
one hopes to find a control $u(\cd) \in \cU[t,T]= L_{\dbF}^2(t,T;\,\dbR^{m})$
to minimize the cost functional \eqref{cost}.

\medskip

Compared with the well-studied classical SLQ optimal control problems
(see the books \cite[Chapter 6]{Yong-Zhou1999} and \cite{Sun-Yong-2020}, for example),
the weighting matrices $Q(t,s)$, $R(t,s)$, and $G(t)$ of \eqref{cost} are  allowed to
depend on the initial time $t$, for which the problem is generally {\it time-inconsistent}.
In other words, an optimal control obtained at a given initial pair might not stay optimal
along the optimal trajectory. This feature distinguishes the current problem from the
classical ones.

\ms

To treat the time-inconsistency, we consider the so-called {\it equilibrium} solution,
which was envisioned by Strotz \cite{Strotz-1955} and then developed by Pollak \cite{Pollak}, Ekland and Lazrak \cite{Ekeland2010},
Hu, Jin, and Zhou \cite{Hu-2012,Hu-2017}, Yong \cite{Yong2012,Yong-2014,Yong-2017},
and Bj\"{o}rk, Khapko, and Murgoci \cite{Bjork-Khapko-Murgoci2017},  to mention a few.

\ms

\begin{definition}\rm
We call $\bar\Th(\cd)\in L^\infty(0,T;\,\dbR^{m\times n})$ a {\it closed-loop equilibrium strategy}, if the following holds:
\begin{equation}
\liminf_{\e\to 0^+} {J(t, \bar X(t);u^\e(\cd))-J(t, \bar X(t);\bar\Th(\cd)\bar X(\cd))\over\e}\geq 0,
\end{equation}
for any $x\in\dbR^n$, $t\in[0,T)$, and  $u(\cd)\in\cU[t,T]=L_\dbF^2(t,T;\,\dbR^m)$,
where $\bar X(\cd)$ is the unique solution of the closed-loop system:
\begin{equation}\label{closed-loop-sys}
\left\{
\begin{aligned}
&d\bar X(s)=[A(s)\bar X(s)+B(s)\bar\Th(s)\bar X(s)]ds  \\
&\qq\qq +[C(s)\bar X(s)+D(s)\bar\Th(s)\bar X(s)]dW(s), \q s\in [0,T],\\
&\bar X(0)=x,
\end{aligned}
\right.
\end{equation}
$X^\e(\cd)$ is the unique solution of the perturbed system:
\begin{equation}
\left\{
\begin{aligned}
&dX^\e (s)=[A(s)X^\e(s)+B(s)u(s)]ds  \\
&\qq\qq~ +[C(s)X^\e(s)+D(s)u(s)]dW(s), \q s\in [t,t+\e],\\
&dX^\e (s)=[A(s)X^\e(s)+B(s)\bar\Th(s)X^\e(s)]ds  \\
&\qq\qq~ +[C(s)X^\e(s)+D(s)\bar\Th(s) X^\e(s)]dW(s), \q s\in (t+\e,T],\\
&X^\e(t)=\bar X(t),
\end{aligned}
\right.
\end{equation}
and
\begin{equation}
u^\e(s)\deq\left\{\begin{aligned}
&u(s),&\q s\in[t,t+\e],\\
& \bar\Th(s)X^\e(s), &\q s\in(t+\e,T].\end{aligned}\right.
\end{equation}
\end{definition}

\ms

The intuition behind this definition is  that at any time $t$, the controller is playing a game
against all his incarnations in the future.  An equilibrium strategy is a solution in  closed-loop form
such that any deviation from it at any time instant will be worse off.
We refer the reader to \cite{Strotz-1955,Ekeland2010,Hu-2012,Yong2012,Bjork-Khapko-Murgoci2017,He-Jiang2021}
for more motivations and details of this definition.

\ms

Following Yong \cite{Yong-2017}, the equilibrium Riccati equation (ERE, for short) associated with the closed-loop equilibrium strategy reads
\begin{equation}\label{ERE}
\left\{\begin{aligned}
&\partial_sP(t,s)+ P(t,s)A_\Theta(s)  +A_\Theta(s)^\top  P(t,s)\\
&\q  +C_\Theta(s)^\top  P(t,s)C_\Theta(s) +Q(t,s) +\Th(s)^\top R(t,s)\Th(s)=0, \q 0\leq t\leq s\leq T,\\
&P(t,T)=G(t), \q 0\leq t\leq T,
\end{aligned}\right.
\end{equation}
where
\begin{align}
&\Theta(s)=-[R(s,s)+D(s)^\top P(s,s)D(s)]^{-1}[B(s)^\top P(s,s)+D(s)^\top P(s,s)C(s)],\label{def-Theta}\\
&A_\Theta(s)=A(s)+B(s)\Theta(s) ,\q C_\Theta(s)= C(s)+D(s)\Theta(s),\q s\in[0,T]. \label{A-Theta}
\end{align}
It has been shown (see \cite{Yong-2017,Lu-2023}) that if the ERE \eqref{ERE} admits a bounded positive semi-definite solution
$P(\cdot,\cdot)$, then the function $\Theta(\cd)$ defined by \eqref{def-Theta} is a closed-loop equilibrium strategy.
Thus, the key to obtaining a closed-loop equilibrium strategy of problem \eqref{state}--\eqref{cost}
lies in proving the solvability of \eqref{ERE}.
The ERE \eqref{ERE} has several new features:

\ms
$\bullet$  Note that $P(\cd,\cd)$ depends on two time variables $(t,s)$.
At the point $(t,s)$, the equation  \eqref{ERE} involves the
value $P(s,s)$, and thus is non-local in the first time variable $t$.

\ms
$\bullet$  Alternatively,  we may also view $t$ as a parameter instead of a variable.
That is, we regard \eqref{ERE} as a system of matrix-valued ODEs with parameter $t$ and solution \textcolor[rgb]{1.00,0.00,0.00}{$\{P(\cd,t)\}_{t\in[0,T]}$}.
Then the ERE \eqref{ERE} is an (uncountably) infinite dimensional system of matrix-valued  ODEs,
which is self interacted through the diagonal term $P(s,s)$.

\ms
\noindent
Thus, ERE \eqref{ERE} differs essentially from the classical Riccati equation
derived in time-consistent SLQ optimal control problems (see \cite{Yong-Zhou1999,Sun-Yong-2020}),
in which the unknown $P(t,s)$ is independent of the first variable $t$.

\ms
Under proper conditions, Yong \cite{Yong-2014,Yong-2017} proved the well-posedness of \eqref{ERE} by applying
the celebrated \textit{multi-person differential game method},
whose central idea lies in the control theory (typically in the dynamic programming method).
More precisely, Yong \cite{Yong-2014,Yong-2017}  divided the time interval $[0,T]$ into $N$ subintervals:
$[t_0,t_1),[t_1,t_2),...,[t_{N-1},t_N]$ with $t_0=0,t_N=T$,
and introduced an $N$-person differential game.  In this game,
Player $k$ controls the system on the subinterval $[t_{k-1},t_k)$, and his/her terminal pair becomes the
initial pair for Player $(k+1)$.
Additionally, all players are aware that each player aims to optimize their own cost functional
and will discount the future cost in their own way, regardless of the fact that the later players will control the system.
Using a back induction method, the cost functional  on $[t_{k-1},t_N]$ for each player $k$ can be transformed to a new one over $[t_{k-1},t_k]$,
resulting in each player facing a time-consistent SLQ optimal control problem.
Then by the dynamic programming method, each player will have a unique optimal control for himself/herself,
which also admits a closed-loop representation by introducing the associated Riccati equation over  the corresponding subinterval.
Therefore, an approximate equilibrium strategy and a family of Riccati equations of this game can be obtained,
which will also converge to an equilibrium strategy and a solution of the ERE respectively, as the mesh size tends to zero.

\ms

The multi-person differential game method provides a useful tool to show the solvability of EREs,
and some further development can be found in Dou and L\"{u} \cite{Dou-2020}
and Lazrak, Wang, and Yong \cite{Lazrak-WY2023}.
However, it might be too complex, especially for readers who are mainly interested in the equation itself
but  not very familiar with control theory. Moreover, in some time-inconsistent problems (see \cite{Lu-2024,Wang-Yong-Zhou-2024} on the time-inconsistent problems for FBSDEs, for example), the multi-person differential game method cannot be directly applied,
because the DPP  does not hold even on the subinterval.
It is then natural to ask:
{\it Is there an alternative method, which is independent of the control theory,
to prove the solvability of the ERE \eqref{ERE} directly?
In particular, can we find a Picard iteration sequence to converge to the solution?}
As pointed out by  Yong and Zhou \cite[Chapter 6]{Yong-Zhou1999} and Chen and Zhou \cite{Chen-2000},
the Picard iteration is usually very useful in computing numerical solutions.

\ms

The idea is pretty natural.
But to achieve this goal, namely to find a proper Picard  iteration sequence,
is by no means trivial. The key idea in  \cite{Yong-Zhou1999,Chen-2000}  is that by linearization,
the Riccati equation can be transformed into a Lyapunov differential equation,
and then by the comparison principle of Lyapunov differential equations,
one can construct a decreasing  sequence $\{P_n(\cd)\}$ with a lower bound.
Consequently,  $\{P_n(\cd)\}$ is exactly a Picard iteration sequence.
However, this idea does not work to the current problem.
The main reason lies in that the ERE \eqref{ERE} is non-local. Thus,
one cannot transform it into a Lyapunov differential equation.
Moreover, the comparison principle of matrix-valued non-local equations is not expected in general.

\ms
Based on the above considerations, we introduce a new Picard iteration approach to the solvability of \eqref{ERE}.
We think our  approach is interesting in its own right and is potentially
useful in more applications. The main idea of our proof can be sketched as follows:

\ms
(i) We  first connect  \eqref{ERE} with a Volterra integro-differential equation (i.e., \eqref{lm2.3-eq1})
and show the equivalence between their well-posedness (i.e., Lemmas \ref{lm2.3}--\ref{lm2.4} and Remark \ref{rm2.1}).
Then we establish a priori estimates for  \eqref{lm2.3-eq1} in the case of smooth coefficients  (i.e., Proposition \ref{pro:priori-est}).
We can see the influence of the monotonic condition \ref{ass:H3} very clearly in this part.

\ms
(ii)
Utilizing the priori estimate, we prove the local solvability of \eqref{ERE} with smooth coefficients by
designing a proper Picard  iteration sequence (i.e., Theorem \ref{thm:local-solvability}).
Indeed, the priori  estimate plays an important role in determining the iteration subset of the solution space
(i.e., \eqref{def-C-star} and \eqref{def-D*}).
We highlight again that the Picard  iteration sequence
can provide a numerical algorithm to compute the solution of \eqref{ERE}.

\ms
(iii) Combining the local solvability and  the priori estimate,
the global solvability of \eqref{ERE} with smooth coefficients can be proved (i.e., Theorem \ref{thm:well-posedness}).

\ms
(iv) By modifying the mollification method, the global solvability of \eqref{ERE} can
be extended to the case with non-smooth coefficients (i.e., Proposition \ref{pro:non-smooth}).

\ms

\ms
The rest of the paper is organized as follows.
In section \ref{subsec:1.1}, we provide a literature review on the closely related topics.
Section \ref{sec:pre} collects some preliminary results, which will be used in the paper.
In section \ref{sec-proof-pri}, we establish the key priori estimate. Section \ref{sec:well-posedness} is divided into two parts:
In section \ref{sec:local}, we prove the local solvability, and
the global solvability is established in section \ref{sec:global}.
Section \ref{sec:non-smooth} is devoted to handle the case with non-smooth coefficients.
Some technical proofs are given in Appendix.

\subsection{Literature review on some related topics}\label{subsec:1.1}

The earliest mathematical consideration of time-inconsistent optimal controls was given by Strotz
\cite{Strotz-1955}, followed by Pollak \cite{Pollak}, and the recent works of Basak and  Chabakauri \cite{Basak-2008},
Ekeland and Pirvu \cite{Ekeland2008}, Ekeland and Lazrak \cite{Ekeland2010}, Yong \cite{Yong2012,Yong-2014,Yong-2017},
Wei, Yong, and Yu \cite{Wei-Yong-Yu2017},  Mei and Zhu \cite{Mei-Zhu2020}, Wang and Yong \cite{Wang-Yong2021},
Hamaguchi \cite{Hamaguchi2021}, Nguwi and Privault \cite{Nguwi}, He and Jiang \cite{He-Jiang2021},
Hern\'{a}ndez and Possama\"{\i} \cite{Hernandez-2023,hernandez-2024},
and  Wang, Yong, and Zhou \cite{Wang-Yong-Zhou-2024} for various kinds of problems.
For the time-inconsistent SLQ problems, Hu, Jin, and Zhou \cite{Hu-2012,Hu-2017} considered the equilibrium solution in the open-loop sense;
Yong \cite{Yong-2017} studied the closed-loop equilibrium strategy by introducing the multi-person differential game method;
Dou and L\"{u} \cite{Dou-2020} extended the result of \cite{Yong-2017} to the infinite-dimensional system
described by a stochastic evolution equation;
Lazrak, Wang, and Yong \cite{Lazrak-WY2023} studied a special SLQ zero-sum game problem motivated by some financial applications;
Cai et al. \cite{Cai-2022} and L\"{u} and Ma \cite{Lu-2023,Lu-2024} gave some equivalent characterizations of the existence of
a closed-loop equilibrium strategy in various cases.
In this paper, we shall give a new proof to the solvability of the ERE
associated with closed-loop equilibrium strategies.
Note that the ERE plays the same role as that of HJB equations in general stochastic optimal control problems.
Thus, its solvability can be regarded as the most important issue  in solving time-inconsistent SLQ problems.

\section{Preliminary}\label{sec:pre}

Let $T>0$ be a given time horizon, $\dbS^n$  be the subspace of $\dbR^{n\times n}$ consisting of symmetric matrices,
and $\dbS_+^n$  be the subset of $\dbS^n$ consisting of positive semi-definite matrices.
For any $0\leq t_1\leq t_2\leq T$, denote
$$
\D^*[t_1,t_2]=\{(t,s)\in[t_1,t_2]^2\hbox{ with } t\leq s\}.
$$
We will use $\cC>0$ to represent a generic constant which could be different from line to line.
For any Euclidean space $\dbH$, we introduce the following spaces:
\begin{align*}
& L^2(0,T;\,\dbH)= \Big\{\f:[0,T]\to\dbH\bigm|\f(\cd)~\hb{is measurable, }\int_0^T|\f(s)|^2ds<\infty \Big\};\\
& L^\infty(0,T;\,\dbH)= \Big\{\f:[0,T]\to\dbH\bigm|\f(\cd)~\hb{is measurable, }\esssup_{0\leq s\leq T}|\f(s)|<\infty \Big\};\\
& C([0,T];\,\dbH)= \Big\{\f:[0,T]\to\dbH\bigm|\f(\cd)~\hb{is continuous, }\sup_{0\leq s\leq T}|\f(s)|<\infty \Big\};\\
& C(\D^*[0,T];\,\dbH)= \Big\{\f:\D^*[0,T]\to\dbH\bigm|\f(\cd,\cd)~\hb{is continuous, }\sup_{0\leq t\leq s\leq T}|\f(t,s)|<\infty \Big\};\\
&C^{\infty}([0,T];\,\dbH)= \Big\{\f:[0,T]\to\dbH\bigm|\f(\cd)~\hb{is infinitely differentiable}\Big\};\\
&C_c^{\infty}(\dbR;\,\dbH)= \Big\{\f:\dbR\to\dbH\bigm|\f(\cd)\in C^{\infty}(\dbR;\,\dbH)~\hb{with compact support}\Big\};\\
& L_{\cF_t}^2(\Om;\,\dbH)
=\Big\{\f:\Om\to\dbH\bigm|\f(\cd)~\hb{is $\cF_t$-measurable},\q \dbE\big[|\f(\omega)|^2\big]<\infty \Big\};\\
& L_\dbF^2(\Om;\,C([0,T];\,\dbH))
=\Big\{\f:[0,T]\times\Om\to\dbH\bigm|\f(\cd)~\hb{is $\dbF$-adapted, pathwise continuous, }\\
&\qq\qq\qq\qq\qq\qq\q \dbE\[\ds\sup_{0\leq s\leq T}|\f(s)|^2\]<\infty \Big\};\\
& L_\dbF^2(0,T;\,\dbH)
=\Big\{\f:[0,T]\times\Om\to\dbH\bigm|\f(\cd)~\hb{is progressively measurable, } \\
&\qq\qq\qq\qq\qq\qq\q  \dbE\[\int_0^T|\f(s)|^2ds\]<\infty \Big\}.
\end{align*}

To guarantee the well-posedness of the controlled state equation \eqref{state}, we introduce the following assumptions.

\begin{taggedassumption}{(H1)}\label{ass:H1}
The coefficients and the weighting matrices satisfy:
\begin{align*}
&A(\cd),C(\cd)\in C([0,T];\, \dbR^{n\times n}), \q B(\cd),D(\cd)\in C([0,T];\, \dbR^{n\times m}),\\
&G(\cd)\in C([0,T];\,\dbS^n),\q   Q(\cd,\cd)\in C([0,T]^2;\, \dbS^{n}),\q R(\cd,\cd)\in C([0,T]^2;\, \dbS^{m}).
\end{align*}
\end{taggedassumption}

\ms

Note that under \ref{ass:H1}, for any given $\Th(\cd)\in L^\infty(t,T;\,\dbR^{m\times n})$,
the closed-loop system \eqref{closed-loop-sys} (with $0$ and $\bar\Th(\cd)$ replaced by $t$ and $\Th(\cd)$, respectively)
admits a unique solution $X(\cd)=X(\cd\,;\,t,x,\Th(\cd))\in L_\dbF^2(\Om;\,C([t,T];\,\dbR^n))$
and the outcome $u(\cd)=\Th(\cd)X(\cd)$ belongs to $\cU[t,T]=L_\dbF^2(t,T;\,\dbR^m)$.
Then the cost functional \eqref{cost} is well-defined.

\ms

To find an equilibrium strategy,  we introduce the following assumptions.

\begin{taggedassumption}{(H2)}\label{ass:H2}
There exists a constant $\delta>0$, such that for any $0\leq t\leq s \leq T$,
\begin{equation*}
	Q(t,s)\geq 0,\q  R(t,s)\geq  \delta I,  \q G(t)\geq 0.
\end{equation*}
\end{taggedassumption}
	
\begin{taggedassumption}{(H3)}\label{ass:H3}
For any $0\leq t\leq r \leq s\leq  T$,
\begin{equation*}
Q(t,s)\leq Q(r,s) \leq \widehat{Q}, \q R(t,s)\leq R(r,s)\leq \widehat{R} ,\q G(t)\leq G(r) \leq \widehat{G},
\end{equation*}
where $\widehat{Q}$, $\widehat{R}$, and $\widehat{G}$ are given positive definite matrices.
\end{taggedassumption}

\ms

We now give some useful lemmas.

\begin{lemma}\label{lm2.2}
Suppose that  $\cR(\cd)\in C([0,T];\,\dbS^{n })$ satisfies  the following monotonic condition:
\begin{equation*}
\cR(t)\leq \cR(s),\q 0\leq t\leq s \leq T.
\end{equation*}
Then $\cR(\cd)$ is differentiable almost everywhere and its derivative is positive semi-definite.
\end{lemma}

\begin{proof}
From the  inequality in the above, we know that $\hbox{Trace}\big(\cR(\cd)\big)$ is a monotonically increasing function and therefore bounded variation. As
\begin{align*}
	\big| \cR(t)-\cR(s) \big|\leq \cC 	\big| \hbox{Trace}\big(\cR(t)\big)- \hbox{Trace}\big( \cR(s)\big) \big|,
\end{align*}
 we get that components of $\cR(\cd)$  are also bounded variation.  Hence   $\cR(\cd)$   is differentiable almost everywhere.
Suppose $\cR(\cd)$ is differentiable at $t_0$, then for any $x\in \dbR^n$, we have
\begin{align*}
x^{\top}\cR'(t_0) x = x^{\top}\lim_{t \to t_0 } \frac{\cR(t)-\cR(t_0)}{t-t_0} x \geq 0,
\end{align*}
which completes the proof.
\end{proof}

\begin{lemma}\label{lm2.1}
Suppose that
\begin{equation*}
	\cA (\cd), \cC(\cd) \in C ([0,T];\,\dbR^{n\times n}),\q \cQ(\cd)\in C ([0,T];\,\dbS^{n}),\q \cG \in \dbS^{n}.
\end{equation*}
Let
\begin{equation*}
{\cP}(t)=\mE \Big[ \Psi(t,T)^\top \cG\Psi(t,T)+ \int_{t}^{T} \Psi(t,s)^\top  \cQ(s)\Psi(t,s)ds   \Big],  \q 0\leq t \leq T,
\end{equation*}
where
\begin{equation*}
\Psi(t,s)=I+\int_{t}^{s} \cA(r)\Psi(t,r)dr+\int_{t}^{s}\cC(r)\Psi(t,r)dW(r),  \q 0\leq t \leq s\leq  T.
\end{equation*}
Then, we have
\begin{equation*}
\dot{\cP}(t)=-\cP(t)\cA(t)-\cA(t)^\top\cP(t)-\cC(t)^\top\cP(t)\cC(t)-\cQ(t),\q t\in [0,T].
\end{equation*}
\end{lemma}

\ms
The proof of Lemma \ref{lm2.1} is standard.
The following provides a representation for $P(\cd,\cd)$ along the ``diagonal'' points.

\ms

\begin{lemma}\label{lm2.3}
Suppose {\rm \ref{ass:H1}--\ref{ass:H3}} hold and  the ERE \eqref{ERE}
admits a  solution $P(\cd,\cd)\in C(\D^*[0,T];\,\dbS^n_+)$.
Then the diagonal values $\{P(t,t);\,t\in[0,T]\}$ satisfies the following Volterra integro-differential equation:
\begin{equation}\label{lm2.3-eq1}
\left\{\begin{aligned}
&P(t,t)=\dbE \Big[ \Phi(t,T)^\top G(t)\Phi(t,T)+ \int_{t}^{T} \Phi(t,s)^\top \big( Q(t,s)+\Theta(s)^\top R(t,s)\Theta(s)\big)   \Phi(t,s)ds   \Big], \\
&\Phi(t,s)=I+\int_{t}^{s} A_{\Th}(r)\Phi(t,r)dr+\int_{t}^{s}C_{\Theta}(r)\Phi(t,r)dW(r),  \q 0\leq t \leq s\leq  T,
\end{aligned}\right.
\end{equation}
where $A_\Th(\cd)$, $C_\Th(\cd)$, and $\Th(\cd)$ are defined by \eqref{def-Theta}--\eqref{A-Theta}.
\end{lemma}

\begin{proof}
Note that for the  positive semi-definite solution $P(\cd,\cd)$, from \ref{ass:H2} we have
$$
R(s,s)+D(s)^\top P(s,s)D(s)\geq R(s,s)\geq \d I,\q s\in[0,T].
$$
Then the corresponding functions $\Th(\cd)$, $A_\Th(\cd)$, and $C_\Th(\cd)$ are bounded.
It follows that  for any given $t\in[0,T)$, the second equation in \eqref{lm2.3-eq1} (which is a linear SDE with bounded coefficients)
admits a unique solution $\Phi(t,\cd)\in L_\dbF^2(\Om;\,C([t,T];\,\dbR^n))$.
Applying  It\^{o}'s formula to the mapping $s\mapsto\Phi(t,s)^\top  P(t,s)\Phi(t,s)$ over $[t,T]$,  we have
\begin{align*}
&{P}(t,t)-\Phi(t,T)^\top G(t)\Phi(t,T)\\
&\q=\int_{t}^{T}\Phi(t,s)^\top \big( Q(t,s)+\Theta(s)^\top R(t,s)\Theta(s)\big)\Phi(t,s)ds \\
&\qq - \int_{t}^{T} \big(\Phi(t,s)^\top {P}(t,s)C_{\Th}(s) \Phi(t,s) + \Phi(t,s)^\top  C_{\Th}(s)^\top {P}(t,s)\Phi(t,s)   \big)dW(s).
\end{align*}
Taking expectation on  both sides of the above, the conclusion is clear.
\end{proof}	

\begin{remark}\rm
Note that \eqref{lm2.3-eq1} is a system of a Volterra integral equation and
an SDE, due to which we call it a  Volterra integro-differential equation.
Moreover, we remark that $P(\cd,\cd)\in C(\D^*[0,T];\,\dbS^n_+)$ means that $P(\cd,\cd)$ is continuous and positive semi-definite.
\end{remark}

Conversely, we also have the following result.

\begin{lemma}\label{lm2.4}
Suppose  {\rm\ref{ass:H1}--\ref{ass:H3}} hold and  the Volterra integro-differential equation \eqref{lm2.3-eq1}
admits a  solution $\{P(t,t);\,t\in[0,T]\}\in C([0,T];\,\dbS^n_+)$.
Then the ERE \eqref{ERE}	admits a  positive semi-definite solution.
\end{lemma}

\begin{proof}

Let $A_\Th(\cd)$, $C_\Th(\cd)$, and $\Th(\cd)$ be defined by \eqref{def-Theta}--\eqref{A-Theta} and we consider the following equations:
\begin{equation}\label{pr-lm2.4-eq1}
\left\{\begin{aligned}
&\partial_s\wt P(t,s)+ \wt P(t,s) A_\Theta(s)  +A_\Theta(s)^\top \wt P(t,s) +C_\Theta (s)^\top\wt  P(t,s)C_\Theta(s) \\
&\q   +Q(t,s) +\Th(s)^\top R(t,s)\Th(s)=0, \q 0\leq t\leq s\leq T,\\
& P(T,t)=G(t), \q 0\leq t\leq T,
\end{aligned}\right.
\end{equation}
and
\begin{equation}\label{pr-lm2.4-eq2}
\wt \Phi(t,s)=I+\int_{t}^{s} A_{\Th}(r)\wt \Phi(t,r)dr+\int_{t}^{s}C_{\Theta}(r)\wt \Phi(t,r)dW(r),  \q 0\leq t \leq s\leq  T.
\end{equation}
By the standard results of ODEs and SDEs, \eqref{pr-lm2.4-eq1} and \eqref{pr-lm2.4-eq2}
admit unique solution $\wt P(\cd,\cd)$ and $\wt\Phi(\cd,\cd)$, respectively.
Since $\Phi(\cd,\cd)$ and $\wt \Phi(\cd,\cd)$ solve the same linear SDE,
we have $\Phi(\cd,\cd)=\wt \Phi(\cd,\cd)$.
By Applying  It\^{o}'s formula to the mapping $s\mapsto\wt\Phi(t,s)^\top \wt P(t,s) \wt \Phi(t,s)$ over $[t,T]$,  we have
\begin{align}
\wt {P}(t,t)&= \mE \[ \wt \Phi(t,T)^\top G(t) \wt \Phi(t,T)+\int_{t}^{T} \wt \Phi(t,s)^\top
\big(  Q(t,s)+\Theta(s)^\top R(t,s)\Theta(s) \big) \wt \Phi(t,s)ds \]\nonumber\\
& =P(t,t),\q t\in[0,T].\label{pr-lm2.4-eq3'}
\end{align}
Thus, the function $\Theta(\cd)$ in  \eqref{lm2.3-eq1} can be rewritten as
\begin{equation}\label{pr-lm2.4-eq3}
\Theta(t)=-\big[R(t,t)+D(t)^\top\wt P(t,t)  D(t)\big]^{-1}  \big[B(t)^\top  \wt P(t,t)   +D(t)^\top\wt P(t,t)  C(t)\big].
\end{equation}
From \eqref{pr-lm2.4-eq1}--\eqref{pr-lm2.4-eq3}, we can see that $\wt P(\cd,\cd)$ is a solution to the equilibrium Riccati equation \eqref{ERE}.
\end{proof}

\begin{remark}\label{rm2.1}\rm
Combining Lemmas \ref{lm2.3} and \ref{lm2.4} together, we know the following two statements are equivalent:
\begin{itemize}
	\item The ERE \eqref{ERE} admits a unique  solution $P(\cd,\cd)\in C(\D^*[0,T];\,\dbS^n_+)$;
    \item The Volterra integro-differential equation \eqref{lm2.3-eq1} admits a unique  solution
    $\{P(t,t);\,t\in[0,T]\}\in C([0,T];\,\dbS^n_+)$.
\end{itemize}
Moreover, from \eqref{pr-lm2.4-eq3'} and \eqref{pr-lm2.4-eq3}, we know that the ERE \eqref{ERE} and the
Volterra integro-differential equation \eqref{lm2.3-eq1} share the same $\Theta(\cd)$.
\end{remark}

The following provides a mollification method to the coefficients, whose proof  is given in Appendix.

\begin{lemma}\label{lm2.5}
Let $\cG(\cd)\in C([0,T];\,\dbS^n)$  and  $\cQ(\cd,\cd)\in C([0,T]^2;\, \dbS^{n})$  satisfy
\begin{equation*}
 0\leq\cG(t)\leq \cG(s)\leq \widehat{G}\,\text{  and  }\,  \delta I \leq\cQ(t,s)\leq \cQ(r,s)\leq \widehat{Q},\q  0\leq t \leq r \leq s \leq T,
\end{equation*}
where $\d>0$ is a fixed constant. Then there exist a sequence $\{\cG_n(\cd)\}_{n\geq 1}\subset C^{\infty}([0,T];\,\dbS^n)$ satisfying
\begin{itemize}
	\item $0 \leq \cG_n(t)\leq \cG_n(s) \leq \widehat{G},\q  0\leq t \leq s \leq T,$
	\item $   \lim_{n\to \infty}\|\cG_n(\cd)-\cG(\cd)\|_{C([0,T];\,\dbS^n)}=0,$
\end{itemize}
and a sequence $\{\cQ_n(\cd,\cd)\}_{n\geq 1}\subset C([0,T]^2;\,\dbS^{n})$ satisfying
\begin{itemize}
	\item $ \text{for any }   0\leq    s \leq T,\q  \cQ_n(t,s) \text{ is infinitely differentiable with respect to $t$, }$
	\item $ \delta I \leq\cQ_n(t,s)\leq \cQ_n(r,s)\leq \widehat{Q},\q  0\leq t \leq r \leq s \leq T,$
	\item $\lim_{n\to \infty}\int_{0}^{T} \sup_{t\in [0,s]}
	\big|  \cQ_n(t,s)  -  \cQ(t,s) \big|^2
	 ds=0.$
\end{itemize}
\end{lemma}

\section{A priori estimates}\label{sec-proof-pri}

To establish the well-posedness of ERE \eqref{ERE},
we need to establish a priori estimates, which require utilizing the differentiability of the mapping $t\mapsto P(t,t)$.
Thus, we assume that $G(\cd)$, $Q(\cd,\cd)$, and $R(\cd,\cd)$ are differentiable.
More precisely, we impose the following additional assumption for the weighting matrices at the moment.

\begin{taggedassumption}{(H4)}\label{ass:H4}
The weighting matrices satisfy:
\begin{align*}
& G(\cd)\in C^1([0,T];\dbS^n), \hbox{ and } Q(t,s) \hbox{ and } R(t,s)
\text{ are continuously  differentiable}\\
& \hbox{for $t$ over } \{(t,s)\, | \,  0\leq t  \leq s \leq T \}.
\end{align*}
\end{taggedassumption}

\begin{proposition}\label{pro:priori-est}
Let {\rm \ref{ass:H1}}--{\rm\ref{ass:H4}} hold. Suppose that the Volterra integro-differential equation \eqref{lm2.3-eq1}
admits a solution $\{P(t,t);\,t\in[0,T]\}\in C([0,T];\,\dbS^n_+)$.
Then we have
\begin{equation}\label{pro:priori-est1}
\sup_{t\in [0,T]}|P(t,t)| \leq 	\big(  |\widehat{G}|+ T  |\widehat{Q}| \big)
e^{ \int_{0}^{T} ( 2|A(s)| + |C(s)|^2)ds}.
\end{equation}
\end{proposition}

\begin{proof}
For any fixed  $t,t_0\in [0,T]$, from \eqref{lm2.3-eq1} we have
\begin{equation*}
\frac{P(t,t)-P(t_0,t_0)}{t-t_0}= \dbI_1+ \dbI_2,
\end{equation*}
where
\begin{align*}
\dbI_1&= \frac{1}{t-t_0}\bigg \{\mE \Big[ \Phi(t,T)^\top G(t)\Phi(t,T)+ \int_{t}^{T} \Phi(t,s)^\top \big( Q(t,s)
+\Theta(s)^\top R(t,s)\Theta(s)\big)\Phi(t,s)ds\Big]\\
&\qq\qq-\mE\[\Phi(t,T)^\top G(t_0)\Phi(t,T)+\int_{t}^{T} \Phi(t,s)^\top \big(Q(t_0,s)+\Theta(s)^\top R(t_0,s)\Theta(s)\big)\Phi(t,s)ds\Big] \bigg\},\\
\dbI_2&=\frac{1}{t-t_0}\bigg \{\mE\[ \Phi(t,T)^\top G(t_0)\Phi(t,T)+ \int_{t}^{T} \Phi(t,s)^\top \big( Q(t_0,s)+\Theta(s)^\top R(t_0,s)
\Theta(s)\big)\Phi(t,s)ds  \Big]\\
&\qq\qq -\mE\[\Phi(t_0,T)^\top G(t_0)\Phi(t_0,T)+ \int_{t_0}^{T} \Phi(t_0,s)^\top \big(Q(t_0,s)
+\Theta(s)^\top R(t_0,s)\Theta(s)\big)\Phi(t_0,s)ds \] \bigg\}.	
\end{align*}
Then for $\dbI_1$, by the dominated convergence theorem and Lemma \ref{lm2.2}, we have
\begin{align*}
\lim_{t \to t_0} \dbI_1 &= \mE \Big[\Phi(t_0,T)^\top G'(t_0)\Phi(t_0,T)+ \int_{t_0}^{T} \Phi(t_0,s)^\top\big(\partial_t Q(t_0,s)\\
&\qq +\Theta(s)^\top\partial_t R(t_0,s)\Theta(s)\big)\Phi(t_0,s)ds\Big] \\
&\geq 0,
\end{align*}
for any $ t_0 \in [0,T]$. We next give an estimate for $\dbI_2$.
For any fixed $t_0 \in [0,T]$, we introduce the following auxiliary function over $[0,T]$:
\begin{equation} \label{de-cP}
\cP (t)\deq\mE\[ \Phi(t,T)^\top G(t_0)\Phi(t,T)+ \int_{t}^{T} \Phi(t,s)^\top
\big(Q(t_0,s)+\Theta(s)^\top R(t_0,s)\Theta(s)\big)\Phi(t,s)ds\],
\end{equation}
where $\Phi(\cd,\cd)$ is uniquely determined by \eqref{lm2.3-eq1}. From \eqref{de-cP}, we can see that
\begin{equation}\label{pr-eq1}
\cP(t_0)= P(t_0,t_0)\q\hbox{and}\q \dbI_2=\frac{\cP(t)-\cP(t_0)}{t-t_0}.
\end{equation}
Moreover, by Lemma \ref{lm2.1},  we have
\begin{equation}\label{pr-eq2}
\dot{\cP}(t) =-\cP(t)A_{\Theta}(t)-A_{\Theta}(t)^\top\cP(t)-C_{\Theta}(t)^\top\cP(t)C_{\Theta}(t)
- Q(t_0,t)-\Theta(t)^\top R(t_0,t)\Theta(t),
\end{equation}
with
\begin{equation}\label{pr-eq3}
\Theta(t)=-\big[R(t,t)+D(t)^\top P(t,t)D(t)\big]^{-1}\big[B(t)^\top P(t,t)+D(t)^\top P(t,t)C(t)\big].
\end{equation}
From \eqref{pr-eq1}--\eqref{pr-eq3}, we know that
\begin{align*}
\dot{\cP}(t_0)
&=-\cP(t_0)A_{\Theta}(t_0)-A_{\Theta}(t_0)^\top\cP(t_0)-C_{\Theta}(t_0)^\top\cP(t_0)C_{\Theta}(t_0)\\
&\q-Q(t_0,t_0)-\Theta(t_0)^\top R(t_0,t_0)\Theta(t_0)\\
&= -P(t_0,t_0) A(t_0)-A(t_0)^\top P(t_0,t_0)-C(t_0)^\top P(t_0,t_0)C(t_0)-Q(t_0,t_0)\\
&\q+\big[ P(t_0,t_0) B (t_0)+  C(t_0)^\top P(t_0,t_0)D(t_0)\big]
\big[R(t_0,t_0)+D(t_0)^\top P(t_0,t_0)  D(t_0)\big]^{-1} \\
& \q \times \big[B(t_0)^\top P(t_0,t_0) +D(t_0)^\top P(t_0,t_0)C(t_0)\big],\q  t_0 \in [0,T].
\end{align*}
Thus for $\dbI_2$, we have
\begin{align*}
		\lim_{t \to t_0} \dbI_2 = \lim_{t \to t_0} \frac{\cP(t)-\cP(t_0)}{t-t_0} =\dot{\cP}(t_0), \q t_0 \in [0,T].
\end{align*}
Consequently, for any $ t_0 \in [0,T]$, we have
\begin{align*}
\lim_{t \to t_0} 		\frac{P(t,t)-P(t_0,t_0)}{t-t_0} & =\lim_{t \to t_0} \dbI_1+ 	\lim_{t \to t_0} \dbI_2\\
& \geq -P(t_0,t_0) A(t_0) -A(t_0)^\top  P(t_0,t_0) -C(t_0)^\top P(t_0,t_0)C(t_0)-Q(t_0,t_0).
\end{align*}
Denote
\begin{equation*}
P'(t_0,t_0)\deq\lim_{t \to t_0}\frac{P(t,t)-P(t_0,t_0)}{t-t_0},\q t_0\in[0,T].
\end{equation*}
Then,
\begin{align*}
P(t,t)-G(T) &=\int_{t}^{T} -P'(s,s)ds \\
&\leq\int_{t}^{T}\big[P(s,s) A(s) +A(s)^\top  P(s,s) +C(s)^\top P(s,s)C(s)+Q(s,s)\big]ds,
\end{align*}
which implies
\begin{align*}
|P(t,t)| &\leq |\widehat{G}| e^{\int_{t}^{T} ( 2|A(s)| + |C(s)|^2)ds}
+ \int_{t}^{T} |\widehat{Q}| e^{\int_{t}^{s}( 2|A(r)|+|C(r)|^2)dr}\\
&\leq \big( |\widehat{G}|+T|\widehat{Q}|\big)e^{\int_{t}^{T} ( 2|A(s)|+ |C(s)|^2)ds},
\end{align*}
where $\widehat{G}$ and $\widehat{Q}$ are given by \ref{ass:H3}. The desired result is proved.
\end{proof}
	
\begin{remark}
\rm From Lemma \ref{lm2.3} and Remark \ref{rm2.1}, the priori estimate \eqref{pro:priori-est1}
is also valid for the ERE \eqref{ERE}.
\end{remark}

\section{Well-posedness of ERE \eqref{ERE}}\label{sec:well-posedness}

In this section, we shall prove the local and global solvability of ERE \eqref{ERE} using the priori estimate (i.e., Proposition \ref{pro:priori-est}).

\ms

To proceed, we first make some preliminary preparations.  For any given $\Th(\cd)\in C([0,T];\,\dbR^{m\times n})$,
we can uniquely determine the function $\{P(t,t);t\in[0,T]\}$ by the following equation:
\begin{equation}\label{lm2.3-eq1'}
\begin{cases}\ds{P}(t,t)=\dbE\[ \Phi(t,T)^\top G(t)\Phi(t,T)+ \int_{t}^{T} \Phi(t,s)^\top \big( Q(t,s)+\Theta(s)^\top R(t,s)\Theta(s)\big)   \Phi(t,s)ds \Big],  \\ \ns \ds
\Phi(t,s)=I+\int_{t}^{s} A_{\Th}(r)\Phi(t,r)dr+\int_{t}^{s}C_{\Theta}(r)\Phi(t,r)dW(r),  \q 0  \leq t \leq s\leq  T,
\end{cases}
\end{equation}
with $A_\Th(\cd)$ and $C_\Th(\cd)$ defined by \eqref{A-Theta} for the given $\Th(\cd)$. Further, we have the following property.

\begin{lemma}\label{lem:ERE-Volterra}
Suppose  {\rm\ref{ass:H1}--\ref{ass:H3}} hold. Then for any given $\Th(\cd)\in C([0,T];\,\dbR^{m\times n})$,
the function $\{P(t,t);t\in[0,T]\}$ defined by \eqref{lm2.3-eq1'} is positive semi-definite and continuous.
\end{lemma}

\begin{proof}
From \ref{ass:H2}, one can easily show that $P(\cd,\cd)\geq 0$.
By the standard estimates of SDEs, we have
\begin{equation}\label{lem:ERE-Proof1}
\sup_{t\in[0,T]}\mE\[\sup_{s\in[t,T]}|\Phi(t,s)|^2\]
\leq \cC(\Th),
\end{equation}
where $\cC(\Th)$ is a general constant depending on $\|\Th(\cd)\|_{C([0,T];\,\dbR^{m\times n})}$
and the coefficients.
For $0\leq t\leq r \leq s \leq T$, we have
\begin{align*}
\mE \big[ | \Phi(t,s)-\Phi(r,s)|^2 \big]
\leq &\,  2  \, \mE \[ \big| \Phi(t,r)-\Phi(r,r)\big|^2
+\Big|\int_r^{s} A_{\Theta}(\t)\big( \Phi(t,\t)-\Phi(r,\t) \big)d\t \Big|^2\\
&\q+\Big|\int_r^{s} C_{\Theta}(\t)\big( \Phi(t,\t)-\Phi(r,\tau) \big)dW(\t) \Big|^2\]\\
\leq&2\mE\[2\Big|\int_{t}^rA_{\Theta}(\t)\Phi(t,\t)d\t\Big|^2
+2\Big|\int_{t}^r  C_{\Theta}(\t)\Phi(t,\t)dW(\t) \Big|^2 \\
&\qq +\Big|\int_r^{s} A_{\Theta}(\t)\big( \Phi(t,\t)-\Phi(r,\t) \big)d\t \Big|^2 \\
&\qq +\Big|\int_r^{s} C_{\Theta}(\t)\big( \Phi(t,\t)-\Phi(r,\tau) \big)dW(\t) \Big|^2\]\\
\leq& \cC(\Th) (r-t)+ \cC(\Th) \int_r^{s} \mE\big[| \Phi(t,\t)-\Phi(r,\t)|^2 \big]d\t,
\end{align*}
which implies
\begin{equation}\label{lem:ERE-Proof2}
\sup_{r \leq s \leq T}\mE \big[ | \Phi(t,s)-\Phi(r,s) |^2\big]\leq \cC(\Th)|t-r|.
\end{equation}
Then for any $0\leq t \leq r \leq T$,
\begin{align*}
|P(t,t)-P(r,r)|&\leq \cC(\Th) \mE \big[ | \Phi(t,T)-\Phi(r,T) |^2\big]^{1\over 2}
\sup_{s\in[0,T]}\mE\big[|\Phi(s,T)|^2\big]^{1\over 2}\\
&\q+\cC(\Th)|G(t)-G(r)|\sup_{s\in[0,T]}\mE\big[|\Phi(s,T)|^2\big] +\cC(\Th)\int_t^r \dbE\big[| \Phi(t,s) |^2\big]ds\\
&\q+\cC(\Th)\int_r^T \dbE\big[| \Phi(t,s)-\Phi(r,s) |^2\big]^{1\over 2}ds\sup_{t\in[0,T]}\mE\[\sup_{s\in[t,T]}|\Phi(t,s)|^2\]^{1\over 2}\\
&\q+\cC(\Th)\int_r^T \big[|Q(t,s)-Q(r,s)|+|R(t,s)-R(r,s) |\big]ds\\
&\q\times\sup_{t\in[0,T]}\mE\[\sup_{s\in[t,T]}|\Phi(t,s)|^2\].
\end{align*}
Substituting \eqref{lem:ERE-Proof1}--\eqref{lem:ERE-Proof2} into the above, we get
\begin{align}
|P(t,t)-P(r,r)|&\leq\cC(\Th)\big(|t-r|^{1\over 2}+|G(t)-G(r)|\big)\nonumber\\
&\q+\cC(\Th)\int_r^T \big[|Q(t,s)-Q(r,s)|+|R(t,s)-R(r,s) |\big]ds.\label{lem:ERE-Proof3}
\end{align}
This completes the proof.
\end{proof}

\begin{remark}\label{rm4.1}\rm
Combining  Lemma \ref{lem:ERE-Volterra} with  {\rm \ref{ass:H1}--\ref{ass:H2}}  together, we can see  that
$$
-[R(s,s)+D(s)^\top P(s,s)D(s)]^{-1}[B(s)^\top P(s,s)+D(s)^\top P(s,s)C(s)],\q s\in [0,T]
$$
is continuous.
One can easily show that the result above still holds for any $\Theta(\cd)\in L^{\infty}(0,T;\,\dbR^{m\times n})$.
\end{remark}

Next, we make the following notations,
\begin{align}
\cC^*&\deq \frac{1}{\delta} \left( \|B(\cd)\|_{C([0,T];\,\dbR^{n \times m})}
+ \|D(\cd)\|_{C([0,T];\,\dbR^{n \times m})}\|C(\cd)\|_{C([0,T];\,\dbR^{n \times n})}  \right) \nonumber\\
&\qq \times\big(|\widehat{G}|+T|\widehat{Q}|\big)	e^{ \int_{0}^{T} ( 2|A(s)| + |C(s)|^2)ds},\label{def-C-star}
\end{align}
where $\widehat{G}$ and $\widehat{Q}$ have been given in \ref{ass:H3}.
Indeed, the uniform constant $\cC^*$ is motivated from the priori estimate  \eqref{pro:priori-est1}. More specifically,
if the Volterra integro-differential equation \eqref{lm2.3-eq1} admits a solution $\{P(t,t);\,t\in[0,T]\}\in C([0,T];\,\dbS^n_+)$,
then the function $\Theta(\cd)$ defined by \eqref{def-Theta} satisfies the following priori estimate:
\begin{equation}\label{pri-the}
	\|\Theta(\cd)\|_{C([0,T];\,\dbR^{m\times n})} \leq \cC^*.
\end{equation}

\ms

Now we are ready to state the main results in this section.

\subsection{Local solvability}\label{sec:local}

We obtain the local solvability of ERE \eqref{ERE} by constructing a proper Picard  iteration sequence,
and utilizing the priori estimate (i.e., Proposition \ref{pro:priori-est}).

\begin{theorem}\label{thm:local-solvability}
Suppose that {\rm \ref{ass:H1}--\ref{ass:H4}} hold. Then there exists a $\k\in(0,T]$ such that
the ERE \eqref{ERE} admits a unique solution $P(\cd,\cd)\in C(\D^*[T-\k,T];\,\dbS^n_+)$.
Moreover, the function $\Theta(\cd)$ defined by \eqref{def-Theta} belongs to $C([T-\k,T];\,\dbR^{m\times n})$.
\end{theorem}

\begin{proof}
By Lemma \ref{lem:ERE-Volterra} and Remark \ref{rm4.1}, we can define the mapping $\Gamma:  C([T-\k,T];\,\dbR^{m\times n}) \to   C([T-\k,T];\,\dbR^{m \times n}) $ as follows
(in which $\k\in (0,T]$ is to be determined later):
For any $\Theta(\cd) \in  C([T-\k,T];\,\dbR^{m \times n})$,
\begin{align}
\G[\Theta(\cd)](t)=-\big[R(t,t)+D(t)^\top P(t,t)  D(t)\big]^{-1}  \big[B(t)^\top  P(t,t)   +D(t)^\top P(t,t)  C(t)\big],\nonumber\\
 t\in [T-\k,T], \label{def-Gamma}
\end{align}
where $P(\cd,\cd)\geq 0$ is given by \eqref{lm2.3-eq1'} with the time interval replaced by $[T-\k,T]$.
Denote
\begin{align}
\Delta_0&\deq \big \{ \Theta\in \dbR^{m\times n}, \hbox{ with } \big | \Theta-[R(T,T)+D(T)^\top G(T)  D(T)]^{-1}\nonumber \\
&\qq\qq\qq\qq\qq\times [B(T)^\top  G(T)   +D(T)^\top G(T)  C(T)] \big  |\leq 2\cC^* \big \}.\label{def-D*}
\end{align}
We claim that  there exists a $\k\in(0,T]$  to  guarantee the following Property (P):
For any  $\Theta(\cd) \in C([T-\k,T];\,\dbR^{m \times n})$
satisfying $\Theta(t) \in \Delta_0$ for $t\in[T-\k,T]$, we have
\begin{equation}\label{pr-th2-eq2}
\G[\Theta (\cd)](t) \in \Delta_0,\q  t \in [T-\k,T].
\end{equation}

\ms

{\bf Step 1.} We shall prove the claim in this step.
For any  $\Theta(\cd) \in C([T-\k,T];\,\dbR^{m\times n})$ satisfying $\Theta(t) \in \Delta_0$  for $ t \in [T-\k,T]$,
noting
\begin{equation}\label{pr-th2-eq3'}
	\big | [R(T,T)+D(T)^\top G(T)  D(T)]^{-1}[B(T)^\top  G(T)   +D(T)^\top G(T)  C(T)] \big  |\leq \cC^*,
\end{equation}
we have
\begin{equation}\label{pr-th2-eq3}
\|\Theta(\cd)\|_{C([T-\k,T];\,\dbR^{m\times n}) }\leq 3 \cC^*.
\end{equation}
Taking this $\Theta(\cd)$ into \eqref{lm2.3-eq1'}, for $t\in [T-\k,T]$, we have
\begin{align}
&\big |\G[\Theta (\cd)](t)-  [R(T,T)+D(T)^\top G(T)  D(T)]^{-1}[B(T)^\top  G(T)
+D(T)^\top G(T)  C(T) ] \big |\nonumber\\
&\q \leq \frac{1}{\delta ^2} \big( |R(t,t)- R(T,T)| +2 | D(t)-D(T)| |P(t,t)| |D(t)|+|D(t)|^2 | P(t,t)-G(T) |\big)\nonumber
\\
& \qq \times \big( |B(t)|+  |D(t)| |C(t)| \big) |P(t,t)| + \frac{1}{\delta}  \big(| B(t)-B(T) | |P(t,t)|+ |B(T)| | P(t,t)-G(T) |\nonumber \\
&\qq  + |D(t)-D(T)| |P(t,t)| |C(t)|  + |D(T)|  |C(t)|   | P(t,t)-G(T) |\nonumber\\
& \qq + |D(T)|  |G(T)|  |C(t)-C(T)|    \big).\label{pr-th2-eq0}
\end{align}
Therefore, to ensure \eqref{pr-th2-eq2}, we only need to estimate
$\|P(\cd,\cd)\|_{C([T-\k,T];\,\dbR^{n \times n})}$ and $|P(t,t)-G(T)|$.
We first handle $\|P(\cd,\cd)\|_{C([T-\k,T];\,\dbR^{n \times n})}$.
From \eqref{lm2.3-eq1'} and \eqref{pr-th2-eq3}, we can get
\begin{equation}\label{pr-th2-eq4}
\sup_{\k\in (0,T]}	\sup_{T-\k \leq t\leq s \leq T}\dbE\big[| \Phi(t,s)|^2\big]\leq \cC(\cC^*),
\end{equation}
and then
\begin{equation}\label{pr-th2-eq7}
\sup_{\k\in (0,T]} \|P(\cd,\cd)\|_{C([T-\k,T];\,\dbR^{n \times n})} \leq \cC(\cC^*),
\end{equation}
where $ \cC(\cC^*)$ is a general constant depending on $\cC^*$ and system parameters.
Next we handle $|P(t,t)-G(T)|$. To this end, we shall estimate $\dbE\big[|\Phi(t,s)-\Phi(r,s) \big|^2\big]$ for preparation.
For $T-\k\leq t\leq r \leq s \leq T$, similar to \eqref{lem:ERE-Proof2}, we have
\begin{equation}\label{pr-th2-eq5}
\sup_{ r \leq s \leq T}\mE \big[ | \Phi(t,s)-\Phi(r,s) |^2\big]\leq \cC(\cC^*) (r-t),
\end{equation}
where $\cC(\cC^*)$ is independent of $\k$.
Then by the same argument as that employed in the proof of \eqref{lem:ERE-Proof3}, we can prove
\begin{align}
|P(t,t)-P(r,r)| &\leq\cC(\cC^*)\big(|t-r|^{1\over 2}+|G(t)-G(r)|\big)\nonumber\\
&\q+\cC(\cC^*)\int_r^T \big(|Q(t,s)-Q(r,s)|+|R(t,s)-R(r,s) |\big)ds.\label{pr-th2-eq17}
\end{align}
In particular,
\begin{equation}\label{pr-th2-eq17'}
|P(t,t)-G(T)| \leq \cC(\cC^*)\big(|t-T|^{1\over 2}+|G(t)-G(T)|\big).
\end{equation}
From  \eqref{pr-th2-eq0}, \eqref{pr-th2-eq7}, \eqref{pr-th2-eq17'}, and  the uniform continuity
of $B(\cd)$, $C(\cd)$, $D(\cd)$, $G(\cd)$, and $R(\cd,\cd)$, we know that there exists a $\k\in (0,T]$,
depending only on system parameters, such that \eqref{pr-th2-eq2} holds.

\ms
{\bf Step 2.} Fix a $\k\in(0,T]$ with Property (P). Our goal is to prove the well-posedness of \eqref{lm2.3-eq1} on $[T-\k,T]$.
For preparation, we define $\G_{[T-h,T]}: C([T-h,T];\,\dbR^{m\times n})\to C([T-h,T];\,\dbR^{m\times n})$  $(h<\k)$
as the restriction of $\G[\,\cd\,]$ on time interval $[T-h,T]$ in the following sense:
For any $\Theta(\cd) \in C([T-h,T];\,\dbR^{m\times n})$,
\begin{align}
\G_{[T-h,T]}[\Theta(\cd)](t)=-\big[R(t,t)+D(t)^\top P(t,t)  D(t)\big]^{-1}  \big[B(t)^\top  P(t,t)   +D(t)^\top P(t,t)  C(t)\big],
\nonumber\\
 t\in [T-h,T],\label{def-Gamma1}
\end{align}
where $P(\cd,\cd)$ is given by \eqref{lm2.3-eq1'} with the time horizon replaced by $[T-h,T]$.
In this step, we show that $\G_{[T-h,T]}[\,\cd\,]$ is a contraction mapping if $h>0$ is small enough.

\ms
Let $\Theta_1(\cd), \Theta_2(\cd)\in C([T-h,T];\,\dbR^{m\times n})$ satisfy the following property:
$$
\Theta_1(t),\Theta_2(t) \in \Delta_0,\q  t\in [T-h,T].
$$
Taking $\Theta_1(\cd)$ and  $\Theta_2(\cd)$ into \eqref{lm2.3-eq1'}, we get
the corresponding $\Phi_1(\cd,\cd)$, $P_1(\cd,\cd)$, $\Phi_2(\cd,\cd)$, and $P_2(\cd,\cd)$.
By \eqref{pr-th2-eq4} and \eqref{pr-th2-eq7}, we have
\begin{equation}\label{pr-th2-eq8}
\sup_{T-h \leq t \leq s \leq T} \mE\big[ |\Phi_i(t,s) |^2 \big]\leq \cC (\cC^*),\q
\sup_{T-h\leq t \leq T}|P_i(t,t) | \leq \cC (\cC^*),\q i=1,2.
\end{equation}
Then by the standard stability estimates of SDEs, we have
\begin{align}
&\sup_{t\in[T-h,T]}\mE \[\sup_{s \in  [t,T]}   \big| \Phi_1(t,s)-\Phi_2(t,s) \big|^2\]\nonumber\\
&\q\leq\cC(\cC^*)  \int_{T-h}^{T}  |\Theta_1(s)-\Theta_2(s)|^2ds\nonumber\\
&\q\leq \cC(\cC^*)h\sup_{s\in[T-h,T]} |\Theta_1(s)-\Theta_2(s)|^2,\label{pr-th2-eq10}
\end{align}
where the constant $\cC(\cC^*)$ is independent of $h$.
Note that
\begin{align}
&P_1(t,t)-P_2(t,t)\nonumber\\
&\q=\dbE\[ [\Phi_1(t,T)-\Phi_2(t,T)]^\top G(t)\Phi_1(t,T)+\Phi_2(t,T)^\top G(t)[\Phi_1(t,T)-\Phi_2(t,T)]\nonumber\\
&\q\qq + \int_{t}^{T} [\Phi_1(t,s)-\Phi_2(t,s)]^\top \big( Q(t,s)+\Theta(s)^\top R(t,s)\Theta(s)\big)\Phi_1(t,s)ds\nonumber\\
&\q\qq   + \int_{t}^{T} \Phi_2(t,s)^\top \big( Q(t,s)+\Theta(s)^\top R(t,s)\Theta(s)\big)[\Phi_1(t,s)-\Phi_2(t,s)]ds\Big].
\label{pr-th2-eq9}
\end{align}
Then by \eqref{pr-th2-eq8}--\eqref{pr-th2-eq10}, we have
\begin{equation}
|P_1(t,t)-P_2(t,t)|\leq \cC(\cC^*)\sqrt{h}\sup_{s\in[T-h,T]} |\Theta_1(s)-\Theta_2(s)|.
\end{equation}
Recall from \eqref{def-Gamma1} the definition of $\G_{[T-h,T]}[\,\cd\,]$.
Then
\begin{align}
&\sup_{t\in[T-h,T]}\big|\G_{[T-h,T]}[\Theta_1(\cd)](t)-\G_{[T-h,T]}[\Theta_2(\cd)](t)\big|\nonumber\\
&\q\leq\cC(\cC^*)\sqrt{h}\sup_{s\in[T-h,T]} |\Theta_1(s)-\Theta_2(s)|.\label{contration}
\end{align}	
Thus, by taking $h>0$ small enough such that $\cC(\cC^*)\sqrt{h}={1\over 2}$, we know that
$\G_{[T-h,T]}[\,\cd\,]$ is a contraction mapping.

\ms
{\bf Step 3.} In this step, we shall prove that $\G_{[T-h,T]}[\,\cd\,]$
has a unique fixed point $\Th^*(\cd)\in C([T-h,T];\,\dbR^{m\times n})$ with $\Th^*(t)\in\D_0$ for $t\in[T-h]$.
Let
$$\Theta_0(s)= \big[R(T,T)+D(T)^\top
G(T)  D(T)\big]^{-1}  \big[B(T)^\top  G(T)   +D(T)^\top
G(T)  C(T)
\big], \q s\in [T-h,T],$$ and
$$
\Theta_{k+1}(s)= \G_{[T-h,T]}[\Theta_k(\cd)](s),\q s\in [T-h,T],
$$
where $h$ has been determined by Step 2.
From  \eqref{pr-th2-eq2}, we know that the sequence $\{\Theta_k(\cd)\}_{k=0}^{\infty}$ have the following property:
\begin{equation}\label{pr-th2-eq6}
\Theta_k(t) \in \Delta_0,\q   t\in [T-h,T].
\end{equation}
Then  from \eqref{contration}, we have
\begin{equation}\label{pr-th2-eq12}
\|\Theta_{k+2}(\cd)-	\Theta_{k+1}(\cd)\|_{C([T-h,T];\,\dbR^{m\times n})} \leq
{1\over 2}\|\Theta_{k+1}(\cd)-	\Theta_{k}(\cd)\|_{C([T-h,T];\,\dbR^{m\times n})}.
\end{equation}
Thus, we know that $\{\Theta_k(\cd)\}_{k=0}^{\infty}$ is a Cauchy sequence in $C([T-h,T];\,\dbR^{m\times n})$.
Let
$$
\Theta^*(\cd)\deq \lim_{k \to \infty} \Theta_k(\cd),
$$
in $C([T-h,T];\,\dbR^{m\times n})$.
By \eqref{pr-th2-eq6}, we can see
$$
{\Theta^*}(t) \in \Delta_0,\q   t\in [T-h,T].
$$
Then
\begin{align*}
&\big\|\G_{[T-h,T]}[\Theta^*(\cd)] - \Theta^*(\cd) \big\|_{C([T-h,T];\,\dbR^{m\times n})} \\
&\q= \lim_{k \to \infty}  \big\|\G_{[T-h,T]}[\Theta^*(\cd)] - \Theta^*(\cd)
-\G_{[T-h,T]}[\Theta_k(\cd)]+\Theta_{k+1}(\cd)  \big\|_{C([T-h,T];\,\dbR^{m\times n})}\\
&\q\leq\lim_{k \to \infty}  \Big\{   \frac{1}{2} 	\|\Theta_{k}(\cd)-	 \Theta^*(\cd)\|_{C([T-h,T];\,\dbR^{m\times n})}
+\|\Theta_{k+1}(\cd)-	\Theta^*(\cd)\|_{C([T-h,T];\,\dbR^{m\times n})} \Big \}\\
&\q=0,
\end{align*}
which means that $\G_{[T-h,T]}[\,\cd\,]$ has a fixed point $\Theta^*(\cd)$.
Let $\{P^*(t,t);t\in[T-h,T]\}$ be uniquely defined by \eqref{lm2.3-eq1'} with taking $\Th(\cd)=\Th^*(\cd)$.
It follows that $\{P^*(t,t);t\in[T-h,T]\}$ is a  solution to \eqref{lm2.3-eq1} over $[T-h,T]$.
On the other hand, from Proposition \ref{pro:priori-est}, \eqref{pri-the}, and \eqref{pr-th2-eq3'}, we know that the $\Theta(\cd)$ defined in \eqref{def-Theta}, corresponding  to any solution of \eqref{lm2.3-eq1},
must satisfy
\begin{equation*}
	 \Theta(t) \in \Delta_0, \q \forall t\in[T-h,T].
\end{equation*}
This, combined  with the fact that $\G_{[T-h,T]}[\,\cd\,]$  is a
contraction mapping for any  $\Theta(\cd) \in C([T-h,T];\,\dbR^{m\times n})$ satisfying $\Theta(t) \in \Delta_0$  for $ t \in [T-h,T]$
(i.e., \eqref{contration}), implies that the solution of \eqref{lm2.3-eq1} is also unique.
Consequently, we prove the well-posedness of  equation \eqref{lm2.3-eq1} on the time interval $[T-h,T]$.

\ms

{\bf Step 4.}
We shall complete the proof by induction in this step.
Next,  for mapping $\G_{[T-2h,T]}[\,\cd\,]$,   we let
$$
\Theta_0(s)=\Theta^*(T-h)\chi_{[T-2h,T-h]}(s)  + \Theta^*(s)\chi_{(T-h,T]}(s), \q s\in [T-2h,T],
$$
and
$$
\Theta_{k+1}(s)= \G_{[T-2h,T]}[\Theta_k(\cd)](s),\q s\in [T-2h,T].
$$
Note that
$$
\Theta_0(s)\in\D_0,\q s\in[T-2h,T],
$$
and then from the claim in Step 1, we have
$$
\Theta_k(s)\in\D_0,\q s\in[T-2h,T],\q k\geq 1.
$$
Thus, for any $k\geq 0$,
\begin{equation}\label{pr-th2-eq3-k}
\|\Theta_k(\cd)\|_{C([T-2h,T];\,\dbR^{m\times n}) }\leq 3 \cC^*.
\end{equation}
From the definition of $\G_{[T-2h,T]}[\,\cd\,]$ and the fact $\G_{[T-h,T]}[\Theta^*(\cd)]=\Theta^*(\cd)$, we observe that
\begin{equation}\label{pr-th2-eq14}
	\Theta_{k}(s)= \Theta^*(s),\q s\in [T-h,T],\q \forall k \geq 0.
\end{equation}
With the uniform estimates by \eqref{pr-th2-eq8}--\eqref{pr-th2-eq9}, we can prove
\begin{align*}
&\big|\Theta_{k+2}(t)-	\Theta_{k+1}(t)\big|\\
&\q= \big|\G_{[T-2h,T]}[\Theta_{k+1}(\cd)](t) - \G_{[T-2h,T]}[\Theta_{k}(\cd)](t)  \big|\\
&\q\leq 	\cC(\cC^*)  \(\int_{t}^{T}  |\Theta_{k+1}(s)-{\Theta_{k}}(s)|^2ds\)^{1/2}\\
&\q\leq 	\cC(\cC^*)  \(\int_{t}^{T-h}  |\Theta_{k+1}(s)-{\Theta_{k}}(s)|^2ds\)^{1/2},
\end{align*}
where $\cC(\cC^*)$ can be the same as that in \eqref{contration} and the last inequality is obtained from \eqref{pr-th2-eq14}.
Then we have
\begin{align*}
&\|\Theta_{k+2}(\cd)-\Theta_{k+1}(\cd)\|_{C([T-2h,T];\,\dbR^{m\times n})}\\
&\q\leq\cC(\cC^*) \sqrt{h}\|\Theta_{k+1}(\cd)-\Theta_{k}(\cd)\|_{C([T-2h,T];\,\dbR^{m\times n})}\\
&\q\leq  \frac{1}{2}\|\Theta_{k+1}(\cd)-	\Theta_{k}(\cd)\|_{C([T-2h,T];\,\dbR^{m\times n})}.
\end{align*}
From this, we can similarly obtain the well-posedness of equation \eqref{lm2.3-eq1} over $[T-2h,T]$.
Inductively,  we obtain the unique solution of  equation \eqref{lm2.3-eq1} on $[T-\k,T]$.
Then by Lemmas \ref{lm2.3}--\ref{lm2.4} and Remark \ref{rm2.1}, we establish the well-posedness of  ERE \eqref{ERE}
on  $\D^*[T-\k,T]$. Moreover, from \ref{ass:H3}, we know that the unique solution $P(\cd,\cd)\geq 0$.

\end{proof}

\subsection{Global solvability}\label{sec:global}

Combining the local solvability (i.e., Theorem \ref{thm:local-solvability}) with the priori estimate
(i.e., Proposition \ref{pro:priori-est}), we can get the solvability of ERE \eqref{ERE} on the global time interval $[0,T]$.

 \begin{theorem}\label{thm:well-posedness}
Suppose that {\rm \ref{ass:H1}--\ref{ass:H4}} hold.
Then for any $T>0$, the ERE \eqref{ERE} admits  a unique solution  $P(\cd,\cd)\in C(\D^*[0,T];\,\dbS^n_+)$.
Moreover, the function $\Theta(\cd)$ defined by \eqref{def-Theta} belongs to $C([0,T];\,\dbR^{m\times n})$.
 \end{theorem}

\begin{proof}
We define the mapping $\G:  C([T-\kappa-\e,T];\,\dbR^{m \times n}) \to    C([T-\kappa-\e,T];\,\dbR^{m \times n}) $ as follows
(in which $\k$ has been given by Theorem \ref{thm:local-solvability} and $\e\in (0,T-\k]$ is to be determined later):
For any $\Theta(\cd)\in C([T-\kappa-\e,T];\,\dbR^{m \times n}) $,
\begin{align*}
\G[\Theta(\cd)](t)=-\big[R(t,t)+D(t)^\top P(t,t)  D(t)\big]^{-1}  \big[B(t)^\top  P(t,t)   +D(t)^\top P(t,t)  C(t)\big],\\
 t\in [T-\kappa-\e,T],
\end{align*}
where $P(\cd,\cd)$ is given by \eqref{lm2.3-eq1'} with the time interval replaced by $[T-\kappa-\e,T]$.
Then from Theorem \ref{thm:local-solvability},  we  have for any $\Theta(\cd)\in  C([T-\kappa-\e,T-\kappa];\,\dbR^{m \times n})$
with $\Th(T-\k)=\Th^*(T-\k)$,
$$
\G\big[\Theta(\cd)\chi_{[T-\k-\e,T-\k]}(\cd)+ \Theta^*(\cd)\chi_{(T-\k,T]}(\cd) \big](t)=  \Theta^*(t),\q  t\in [T-\k,T],
$$
where $\Th^*(\cd)$ is the unique fixed point determined by Theorem \ref{thm:local-solvability}.
Recall from \eqref{def-C-star} the definition of $\cC^*$.
Denote
$$
\Delta_1\deq \big \{ \Theta\in \dbR^{m\times n},\q\hbox{with }  \big | \Theta- \Theta^*(T-\k) \big  |\leq 2\cC^* \big \}.
$$
From the priori estimate \eqref{pro:priori-est1} in Proposition \ref{pro:priori-est}, we see that
\begin{align}
|\Theta^*(T-\k)|&=\big|-[R(T-\k,T-\k)+D(T-\k)^\top P(T-\k,T-\k)  D(T-\k)]^{-1}\nonumber\\
&\qq\times  [B(T-\k)^\top  P(T-\k,T-\k)  +D(T-\k)^\top P(T-\k,T-\k)  C(T-\k)]\big|\nonumber\\
&\leq\cC^*.\label{Th(T-k)}
\end{align}
Choosing $\e=\k$, we claim that for any  $\Theta(\cd) \in  C([T-2\kappa,T-\kappa];\,\dbR^{m \times n})$ satisfying $\Th(T-\k)=\Th^*(T-\k)$
and $\Theta(t) \in \Delta_1$ for $t \in [T-2\kappa,T-\kappa]$,
we have
\begin{equation}\label{pr-th2-eq16}
\G\big[\Theta(\cd)\chi_{[T-2\k,T-\k]}(\cd)+ \Theta^*(\cd)\chi_{(T-\k,T]}(\cd) \big](t)\in \Delta_1,\q   t \in [T-2\kappa,T-\kappa].
\end{equation}
Indeed, using the fact $\Theta(t) \in \Delta_1$  for $ t \in[T-2\k,T-\k]$, \eqref{Th(T-k)}, and Proposition \ref{pro:priori-est},
we have
\begin{align}
\nonumber&\big\|\Theta(\cd)\chi_{[T-2\k,T-\k]}(\cd)+ \Theta^*(\cd)\chi_{(T-\k,T]}(\cd)\big\|_{C([T-2\k,T];\,\dbR^{m \times n}) }\\
&\q\nonumber \leq \max\big\{  \|\Theta(\cd)\|_{C([T-2\k,T-\k];\,\dbR^{m\times n}) },\,\, \| \Theta^*(\cd)\|_{C([T-\k,T];\,\dbR^{m \times n}) } \big\}\\
&\q\leq \max \{ 3 \cC^*,\,\cC^* \}\nonumber\\
&\q= 3\cC^*.\label{pr-th2-eq18}
\end{align}
Taking $\Theta(\cd)\chi_{[T-2\k,T-\k]}(\cd)+ \Theta^*(\cd)\chi_{(T-\k,T]}(\cd)$ into \eqref{lm2.3-eq1'},
for $t\in [T-2\k,T-\k]$, we have
\begin{align}
&\big |\G\big[\Theta(\cd)\chi_{[T-2\k,T-\k]}(\cd)+ \Theta^*(\cd)\chi_{(T-\k,T]}(\cd) \big](t) - \Theta^*(T-\k)\big |\nonumber\\
&\q \leq \frac{1}{\delta ^2} \big( |R(t,t)- R(T-\k,T-\k)| +2 | D(t)-D(T-\k) | |P(t,t)| |D(t)| 		\nonumber\\
& \qq   +|D(t)|^2 | P(t,t)-P(T-\k,T-\k) |  \big )\big( |B(t)|+  |D(t)| |C(t)| \big) |P(t,t)|\nonumber  \\
& \qq + \frac{1}{\delta} \big(   | B(t)-B(T-\k) | |P(t,t)|+|B(T-\k)| | P(t,t)-P(T-\k,T-\k)|  \nonumber  \\	
&\qq + |D(t)-D(T-\k)| |P(t,t)| |C(t)|   + |D(T-\k)|  |C(t)|    | P(t,t)-P(T-\k,T-\k) | \nonumber  \\
& \qq + |D(T-\k)|  |P(T-\k)|  |C(t)-C(T-\k)|   \big).
\end{align}
Note that with  \eqref{pr-th2-eq18}, the estimate \eqref{pr-th2-eq17} still holds on $[T-2\k,T-\k]$  for the same $\cC(\cC^*)$.
Combining this with the equicontinuity of the coefficients, we know that the claim above still holds true.
It follows that
\begin{equation*}
\big\|\G\big[\Theta(\cd)\chi_{[T-2\k,T-\k]}(\cd)+ \Theta^*(\cd)\chi_{(T-\k,T]}(\cd) \big](\cd)\big\|_{ C([T-2\kappa,T-\kappa];\,\dbR^{m \times n})}\leq 3\cC^*,
\end{equation*}
holds for any  $\Theta(\cd) \in  C([T-2\kappa,T-\kappa];\,\dbR^{m \times n})$ satisfying $\Th(T-\k)=\Th^*(T-\k)$
and $\Theta(t) \in \Delta_1$ for $t \in [T-2\kappa,T-\kappa]$.
Then by applying the same procedure in the proof of Theorem \ref{thm:local-solvability},
we  can establish the well-posedness of  equation \eqref{lm2.3-eq1}   on $[T-2\k,T]$ piecewisely.
Repeating this process,  the well-posedness of  equation \eqref{lm2.3-eq1} can be extended to the global time interval $[0,T]$.
Finally, using Lemmas \ref{lm2.3}--\ref{lm2.4} and Remark \ref{rm2.1} again, we establish the well-posedness of  ERE \eqref{ERE} on  $\D^*[0,T]$.
\end{proof}

\begin{remark}\rm
Both Theorems \ref{thm:local-solvability} and \ref{thm:well-posedness} are proved by constructing a Picard iteration sequence, which
actually provides a numerical algorithm to compute the solution of \eqref{ERE}.
\end{remark}

\section{The case with non-smooth coefficients}\label{sec:non-smooth}

In this section, we will remove the additional differentiability condition \ref{ass:H4} of the weighting matrices.
Indeed, we can replace \ref{ass:H1} and  \ref{ass:H4} by the following weaker assumptions:
\begin{taggedassumption}{(H1)$^*$}\label{ass:H1'}
The coefficients and the weighting matrices satisfy:
\begin{align*}
&A(\cd),C(\cd)\in L^{\infty}(0,T;\, \dbR^{n\times n}),
\q B(\cd),D(\cd)\in L^{\infty}(0,T;\, \dbR^{n\times m}),\\
&G(\cd)\in C([0,T];\,\dbS^n),\q   Q(\cd,\cd)\in C([0,T]^2; \,\dbS^{n}),
\q R(\cd,\cd)\in C([0,T]^2; \,\dbS^{m}).
\end{align*}
\end{taggedassumption}

\ms

\begin{proposition}\label{pro:non-smooth}
Let {\rm\ref{ass:H1'},  \ref{ass:H2}}, and {\rm \ref{ass:H3}}  hold.
Then the ERE \eqref{ERE} admits  a unique solution  $P(\cd,\cd)\in C(\D^*[0,T];\,\dbS^n_+)$.
Moreover, the function $\Theta(\cd)$ defined by \eqref{def-Theta} belongs to $L^\infty([0,T];\,\dbR^{m\times n})$.
\end{proposition}

\begin{proof}
We only need to prove that  the Volterra integro-differential equation \eqref{lm2.3-eq1} admits a unique solution.
For each $A(\cd),C(\cd)\in L^{\infty}(0,T;\, \dbR^{n\times n})$ and $B(\cd),D(\cd)\in L^{\infty}(0,T;\, \dbR^{n\times m})$,
we can select sequences $\{A_n(\cd),C_n(\cd)\}_{n\geq 1 } \subset C^{\infty}([0,T];\, \dbR^{n\times n})$ and $\{B_n(\cd),D_n(\cd)\}_{n\geq 1 } \subset C^{\infty}([0,T];\, \dbR^{n\times m})$ such that
\begin{align*}
&\lim_{n\to \infty} \|A_n(\cd)-A(\cd) \|_{L^2(0,T;\, \dbR^{n\times n})}=0,\\
& 	\lim_{n\to \infty}  \|C_n(\cd)-C(\cd)   \|_{L^2(0,T;\, \dbR^{n\times n})}=0,\\
& \lim_{n\to \infty}  \|B_n(\cd)-B (\cd) \|_{L^2(0,T;\, \dbR^{n\times m})}=0,\\
&\lim_{n\to \infty}  \|D_n(\cd)-D(\cd)  \|_{L^2(0,T;\, \dbR^{n\times m})}=0, \\
&  \sup_{n\geq 1} \|A_n(\cd)  \|_{L^{\infty}(0,T;\,\dbR^{n\times n})}\leq  2 \|A(\cd)  \|_{L^{\infty}(0,T;\, \dbR^{n\times n})} ,\\
& \sup_{n\geq 1} \|C_n(\cd)  \|_{L^{\infty}(0,T;\, \dbR^{n\times n})} \leq 2 \|C(\cd)  \|_{L^{\infty}(0,T;\, \dbR^{n\times n})},\\
& \sup_{n\geq 1} \|B_n(\cd)  \|_{L^{\infty}(0,T;\, \dbR^{n\times m})} \leq 2 \|B(\cd)  \|_{L^{\infty}(0,T;\, \dbR^{n\times m})} ,\\
& \sup_{n\geq 1} \|D_n(\cd)  \|_{L^{\infty}(0,T;\, \dbR^{n\times m})}\leq   2 \|D(\cd)  \|_{L^{\infty}(0,T;\, \dbR^{n\times m})}.
\end{align*}
Moreover by Lemma \ref{lm2.5}, there exist $\{G_n(\cd)\}_{n\geq 1} \subset  C^{\infty}([0,T];\,\dbS^n)$, $\{Q_n(s,t)\}_{n\geq 1}\subset C([0,T]^2;\, \dbS^{n})$, and
  $\{R_n(s,t)\}_{n\geq 1}\subset C([0,T]^2;\, \dbS^{m})$ such that
\begin{itemize}
	\item  $ 0 \leq G_n(t)\leq G_n(r) \leq \widehat{G},\q  0\leq t \leq r \leq T,$
	\item  $  \lim_{n\to \infty}\|G_n(\cd)-G(\cd)\|_{C([0,T];\,\dbS^n)}=0,$
	\item  $ \text{for any } 0\leq s \leq T, \,  Q_n(t,s) \text{ and } R_n(t,s)  \text{ are infinitely differentiable for $t$,}$
	\item  $0 \leq Q_n(t,s)\leq Q_n(r,s) \leq \widehat{Q},  \q  0\leq t \leq r \leq s \leq T,$
	\item  $\delta I \leq R_n(t,s)\leq R_n(r,s) \leq \widehat{R},\q  0\leq t \leq r \leq s \leq T,$
	\item  $\lim_{n \to \infty} \int_{0}^{T} \sup_{t\in [0,s]} \big|  Q_n(t,s)  -  Q(t,s)  \big|^2 ds=0,$
	\item  $\lim_{n \to \infty} \int_{0}^{T} \sup_{t\in [0,s]} \big|  R_n(t,s)  - R(t,s)  \big|^2 ds=0.$
\end{itemize}
For any fixed $n$,  by Theorem \ref{thm:well-posedness},
we know that the following Volterra integro-differential equation admits a unique solution:
\begin{equation}\label{pr4.1-pf-eq2}
\begin{cases}\ds
{P_n}(t,t)=\mE \Big[ \Phi_n(t,T)^\top G_n(t)\Phi_n(t,T)+ \int_{t}^{T} \Phi_n(t,s)^\top \big( Q_n(t,s)\\
\ds\qq\qq\qq+\Theta _n(s)^\top
R_n(t,s)\Theta_n(s)\big) \Phi_n(t,s)ds   \Big],  \\
\ns \ds\Phi_n(t,s)=I+\int_{t}^{s} \big[A_n(r)+B_n(r)\Theta_n(r)\big]\Phi_n(t,r)dr\\
\ds\qq\qq\qq+\int_{t}^{s} \big[C_n(r)+D_n(r)\Theta_n(r)\big] \Phi_n(t,r)dW(r),
\end{cases}
\end{equation}
where
\begin{equation}
\Theta_n(t)=-\big[R_n(t,t)+D_n(t)^\top P_n(t,t)  D_n(t)\big]^{-1}  \big[B_n(t)^\top  P_n(t,t)   +D_n(t)^\top
P_n(t,t)  C_n(t)\big].
\end{equation}
Note that $A_n(\cd)$, $B_n(\cd)$, $C_n(\cd)$, $D_n(\cd)$,
$G_n(\cd)$, $Q_n(\cd,\cd)$, and $R_n(\cd,\cd)$ are uniformly bounded.
Then from  Proposition \ref{pro:priori-est}, we know that $P_n(\cd,\cd)$ are uniformly bounded, namely
\begin{equation*}
\sup_{n\geq 1}\sup_{t\in [0,T]}|P_n(t,t)|
\leq \big(  |\widehat{G}|+ T  |\widehat{Q}|  \big)e^{ \int_{0}^{T} ( 4|A(s)| + 4|C(s)|^2)ds}.
\end{equation*}	
From this, we can deduce that
\begin{equation*}
\sup_{n \geq 1}\| \Theta_n(\cd) \|_{C([0,T];\, \dbR^{m\times n})} \leq  \cC(\cC^*),\q  		
\sup_{n \geq 1}\sup_{0\leq t\leq s \leq T} \mE \big[| \Phi_n(t,s)|^2 \big] \leq    \cC(\cC^*),
\end{equation*}
and
\begin{align}
\big| \Theta_n(t)-\Theta_m(t) \big| &\leq   \cC(\cC^*)  \big( | P_n(t,t)-P_m(t,t) |+|R_n(t,t)-R_m(t,t)|
+| B_n(t)-B_m(t)|\nonumber\\
& \qq\qq + | C_n(t)-C_m(t)|+ | D_n(t)-D_m(t)| \big),\label{thm:modi-1}
\end{align}
where $\cC(\cC^*)$ is a general constant depending on $\cC^*$ and system parameters.
This together with the standard estimates for SDEs gives
\begin{align}
&\mE\[\sup_{s\in [t,T]} \big| \Phi_n(t,s)-  \Phi_m(t,s) \big|^2\]\nonumber\\
&\q\leq \cC(\cC^*) \int_{t}^T\[ |A_n(s)-A_m(s)|^2+ |B_n(s)-B_m(s)|^2+ |C_n(s)-C_m(s)|^2\nonumber\\
&\qq\qq\qq\q + |D_n(s)-D_m(s)|^2+ |\Theta_n(s)-\Theta_m(s)|^2 \]ds\nonumber\\
&\q\leq \cC(\cC^*) \int_{t}^T\[ |A_n(s)-A_m(s)|^2+ |B_n(s)-B_m(s)|^2+ |C_n(s)-C_m(s)|^2\nonumber\\
&\qq\qq\qq\q + |D_n(s)-D_m(s)|^2+|R_n(s,s)-R_m(s,s)|^2+ |P_n(s,s)-P_m(s,s)|^2 \]ds.\label{thm:modi-2}
\end{align}
Then we have
\begin{align*}
&\big| P_n(t,t)-  P_m(t,t) \big|\\
&\q \leq \cC(\cC^*) \Big\{   \mE \big[|  \Phi_n(t,T)-  \Phi_m(t,T) |^2\big]^{1\over 2}+ |G_n(t)-G_m(t)|\\
& \qq+\int_{t}^{T} \[  \mE\big[| \Phi_n(t,s)-  \Phi_m(t,s)|^2 \big]^{1\over 2} + |\Theta_n(s)-\Theta_m(s) | \]  ds\\
&\qq+ \int_{t}^{T} \[   \big|R_n(t,s)-R_m(t,s) \big|_{} + \big|Q_n(t,s)-Q_m(t,s) \big| \]  ds\Big\}\\
&\q\leq  \cC(\cC^*) \Big\{ \|G_n(\cd)-G_m(\cd)\|_{C([0,T];\,\dbS^n)}+ \int_{0}^{T}  \sup_{t\in [0,s]} |   Q_n(t,s)-  Q_m(t,s)|  ds\\
& \qq+ \int_{0}^{T} \sup_{t\in [0,s]} |   R_n(t,s)-  R_m(t,s)|ds
+  \mE \big[|  \Phi_n(t,T)-  \Phi_m(t,T)|^2 \big]^{1\over2}\\
&\qq +\int_{t}^{T} \[ \mE   \big[|  \Phi_n(t,s)-  \Phi_m(t,s) |^2 \big]^{1\over 2} + |\Theta_n(s)-\Theta_m (s) |\]  ds\Big\}.
\end{align*}
Substituting \eqref{thm:modi-1} and \eqref{thm:modi-2} into the above,
we get
\begin{align*}
& \big| P_n(t,t)-  P_m(t,t) \big|^2\\
& \q\leq \cC(\cC^*) \Big\{ \|G_n(\cd)-G_m(\cd)\|^2_{C([0,T];\,\dbS^n)}
+ \(\int_{0}^{T} \sup_{t\in [0,s]} \big|   Q_n(t,s)-  Q_m(t,s)  \big|^2  ds\) \\
&  \qq\q+  \( \int_{0}^{T} \sup_{t\in [0,s]} \big|   R_n(t,s)-  R_m(t,s)  \big|^2 ds\)
+  \int_{t}^{T} \( |A_n(s)-A_m(s)|^2\\
&\qq\q + |B_n(s)-B_m(s)|^2+ |C_n(s)-C_m(s)|^2 + |D_n(s)-D_m(s)|^2 \\
&\qq\q+|P_n(s,s)-P_m(s,s)|^2 \)ds\Big\}.
\end{align*}
By Gr\"{o}nwall's inequality, we can get
\begin{align*}
& \big\| P_n(t,t)-  P_m(t,t) \big\|^2_{ C([0,T];\, \dbR^{n\times n})}\\
& \q\leq \cC(\cC^*) \Big\{ \|G_n(\cd)-G_m(\cd)\|^2_{C([0,T];\,\dbS^n)}
+ \(\int_{0}^{T} \sup_{t\in [0,s]} \big|   Q_n(t,s)-  Q_m(t,s)  \big|^2  ds\) \\
&  \qq+  \( \int_{0}^{T} \sup_{t\in [0,s]} \big|   R_n(t,s)-  R_m(t,s)  \big|^2 ds\)
+  \int_{0}^{T} \( |A_n(s)-A_m(s)|^2\\
&\qq + |B_n(s)-B_m(s)|^2+ |C_n(s)-C_m(s)|^2 + |D_n(s)-D_m(s)|^2  \)ds\Big\}.
\end{align*}
Thus $\{P_n(t,t);\,t\in[0,T]\}_{n\geq 1}$ is a Cauthy sequence in $C([0,T];\, \dbR^{n\times n})$.
Taking $n \to \infty$ in \eqref{pr4.1-pf-eq2},  we obtain one solution for equation \eqref{lm2.3-eq1} for the case:
$A(\cd),C(\cd)\in L^{\infty}(0,T;\, \dbR^{n\times n})$, $B(\cd),D(\cd)\in L^{\infty}(0,T;\, \dbR^{n\times m})$,
 $	G(\cd)\in C([0,T];\,\dbS^n)$, $Q(\cd,\cd)\in C([0,T]^2;\, \dbS^{n})$, and   $R(\cd,\cd)\in C([0,T]^2; \,\dbS^{m})$.
The uniqueness of the  solution for  equation \eqref{lm2.3-eq1}  is standard.
\end{proof}

\section*{Appendix}

\emph{\textbf{Proof of Lemma \ref{lm2.5}.}}
Let $\eta(\cd)$ be the standard mollifier,
\begin{equation}
\eta (t)=\begin{cases}
\cC_0 e^{\frac{1}{t^2-1}}, &|t|\leq 1,\\
0,& |t|>1,
\end{cases}
\end{equation}
where $\cC_0$ satisfies $\int_{\dbR}\eta(t)dt=1$.  Denote $\eta_{1/n}(t)\deq  n \eta(nt)$.

\ms
	
First we  extend the domain  of $\cG(\cd)$ to $\dbR$. Let
\begin{equation*}
\bar{\cG}(t)=\begin{cases}
\int_{\dbR}\cG(0)\chi_{[-2T,0]}(s)\eta_{T/4}(t-s)ds,\q & t\in (-\infty,-T],\\
\cG(0),\q & t\in [-T,0],\\
\cG(t),\q & t\in [0,T],\\
\cG(T),\q & t\in [T,2T],\\
\int_{\dbR}\cG(T)\chi_{[T,3T]}(s)\eta_{T/4}(t-s)ds,\q & t\in [2T,+\infty).
\end{cases}
\end{equation*}
We know that $\bar{\cG}(\cd)$ belongs to $C_c(\dbR;\,\dbS^n)$ and is monotonically increasing over $[-T,2T]$.
Set
\begin{equation}\label{pr-lm2.5-eq1}
\cG_n(t)\= \int_{\dbR} \bar{\cG}(s)\eta_{1/n} (t-s)ds,\q t\in \dbR.
\end{equation}
We know that (see \cite[pp. 713--714]{Evans-2010}, for example) $\cG_n(\cd) \in C_c^{\infty}(\dbR;\,\dbS^n)$ and
\begin{equation*}
\lim_{n\to \infty}\|\cG_n(\cd)-\cG(\cd)\|_{C([0,T];\,\dbS^n)} =0.
\end{equation*}
Since for any $x\in\dbR^n$ and $t\in [0,T]$, it holds that
\begin{align*}
0&\leq x^\top 	\cG_n(t) x =  \int_{[-1/n,1/n]}  x^\top \bar{\cG}(t-s) x\, \eta_{1/n} (s)ds\\
&\leq  \int_{\dbR}  x^\top \widehat{G} x\, \eta_{1/n} (t-s)ds= x^\top 	\widehat{G} x,
\end{align*}
we can get
\begin{equation*}
	0\leq\cG_n(t) \leq \widehat{G},\q t\in [0,T].
\end{equation*}
Besides,  for any $x\in \dbR^n$ and $0\leq t_1\leq t_2 \leq T$,   the following holds for large $n$,
\begin{align}
&x^\top \cG_n(t_2) x-x^\top \cG_n(t_1) x\nonumber\\
&\q= \int_{[-1/n,1/n]} x^\top \big[\bar{\cG}(t_2-s) -\bar{\cG}(t_1-s)\big] x\,  \eta_{1/n} (s)ds \nonumber\\
&\q\geq 0.\label{app-eq1}
\end{align}
This implies that
\begin{equation*}
	\cG_n(t)\leq \cG_n(s),\q 0\leq t\leq s \leq T.
\end{equation*}

\ms

Next we extend the domain of $\cQ(\cd,\cd)$ to $\{(t,s)\,| \, 0\leq s \leq T, -\infty < t < +\infty\}$.
Let
\begin{align*}
	&\bar{\cQ}(t,s)\=\begin{cases}
		\int_{\dbR}\cQ(0,s)\chi_{[-2T,0]}(r)\eta_{T/4}(t-r)dr,&  (t,s)\in \{  (t,s)\,|\, 0\leq s \leq T,-\infty <t\leq  -T \},\\
		\cQ(0,s),&    (t,s)\in \{  (t,s)\,|\, 0\leq s \leq T,-T \leq t\leq 0 \},\\
	    \cQ(t,s), & (t,s)\in \{  (t,s)\,|\, 0\leq s \leq T,0\leq t \leq s \},\\
		\cQ(s,s), &   (t,s)\in \{  (t,s)\,|\, 0\leq s \leq T,s \leq t\leq 2T \},\\
        \int_{\dbR}\cQ(s,s)\chi_{[T,3T]}(r)\eta_{T/4}(t-r)dr,&  (t,s)\in \{  (t,s)\,|\, 0\leq s \leq T,2T \leq t< +\infty \}.
	\end{cases}
\end{align*}
Similar to the case of $\cG(\cd)$,  for any fixed $s\in [0,T]$,
we know that $\bar{\cQ}(\cd,s)$ belongs to $C_c(\dbR;\,\dbS^n)$ and is monotonically increasing over $[-T,2T]$.
Set
\begin{equation*}
	\nonumber\cQ_n(t,s)\= \int_{\dbR} \bar{\cQ}(r,s)\eta_{1/n} (t-r)dr,\q  (t,s) \in  \dbR\times [0,T].
\end{equation*}
We know that, for any fixed $s\in [0,T]$,  $\cQ_n(\cd,s) \in C_c^{\infty}(\dbR;\,\dbS^n)$ and
\begin{equation*}
	\lim_{n\to \infty}\sup_{t\in [0,s]}|\cQ_n(t,s)-\cQ(t,s)| =0,\q \forall s\in [0,T].
\end{equation*}
Similar to \eqref{app-eq1}, we have
\begin{equation*}
	\cQ_n(t,s)\leq \cQ_n(r,s),\q 0\leq t \leq r \leq s \leq T.
\end{equation*}
Note for any $x\in\dbR^n$ and $0\leq t  \leq s \leq T$, it holds that
\begin{align*}
x^\top	\delta I  x &\leq \int_{[-1/n,1/n]}  x^\top \bar{\cQ}(t-r,s) x\, \eta_{1/n} (r)dr\\
&\leq  \int_{\dbR}  x^\top \widehat{Q} x\, \eta_{1/n} (t-s)ds\\
&= x^\top 	\widehat{Q} x.
\end{align*}
Therefore
\begin{equation*}
	\delta I \leq\cQ_n(t,s) \leq \widehat{Q},\q 0\leq t  \leq s \leq T.
\end{equation*}
Consequently, by the dominated convergence theorem, we have
\begin{align*}
\lim_{n\to \infty}\int_{0}^{T} \sup_{t\in [0,s]}
\big|  \cQ_n(t,s)  -  \cQ(t,s) \big|^2
 ds=0.
\end{align*}


\begin{thebibliography}{999}
\bibitem{Basak-2008}  S. Basak and G. Chabakauri,
\it Dynamic mean-variance asset allocation,
\rm Rev.   Finan. Stu., {\bf23} (2010), 2970--3016.


\bibitem{Bjork-Khapko-Murgoci2017} T.~Bj\"{o}rk, M.~Khapko, and A.~Murgoci,
\it On time-inconsistent stochastic control in continuous time,
\rm Finance Stoch., {\bf 21} (2017), 331--360.



\bibitem{Bjork-Murgoci-Zhou2014}  T.~Bj\"{o}rk, A.~Murgoci, and X.~Y.~Zhou,
\it Mean-variance portfolio optimization with state-dependent risk aversion,
\rm Math. Finance, {\bf 24} (2014),  1--24.

\bibitem{Cai-2022} H. Cai, D. Chen, Y. Peng, and W. Wei,
\it On the time-inconsistent deterministic linear-quadratic control,
\rm SIAM J. Control Optim., {\bf60} (2022), 968--991.

\bibitem{Chen-2000} S. Chen and X. Y. Zhou,
\it Stochastic linear quadratic regulators with indefinite control weight costs. II,
\rm SIAM J. Control Optim., {\bf39} (2000),  1065--1081.

\bibitem{Dou-2020}  F. Dou and Q. L\"u,
\it Time-inconsistent linear quadratic optimal control problems for stochastic evolution equations,
\rm SIAM J. Control Optim., {\bf58} (2020), 485--509.

\bibitem{Ekeland2008} I.~Ekeland and T.~A.~Pirvu,
\it Investment and consumption without commitment,
\rm Math. Financ. Econ., {\bf 2} (2008), 57--86.

\bibitem{Ekeland2010} I.~Ekeland and A.~Lazrak,
\it The golden rule when preferences are time inconsistent,
\rm Math. Financ. Econ., {\bf 4} (2010), 29--55.


\bibitem{Evans-2010} L.~C. Evans,
\it Partial differential equations, second edition,
\rm American Mathematical Society, Providence, RI, 2010.

\bibitem{Hamaguchi2021} Y.~Hamaguchi,
\it Extended backward stochastic Volterra integral equations and their applications to time-inconsistent stochastic recursive control problems,
\rm Math. Control Relat. Fields, {\bf11} (2021), 433--478.

\bibitem{He-Jiang2021} X. D. He and Z. L. Jiang,
\it On the equilibrium strategies for time-inconsistent problems in continuous time,
\rm  SIAM J. Control Optim., {\bf59} (2021),  3860--3886.

\bibitem{Hernandez-2023} C.~Hern\'{a}ndez and D.~Possama\"{\i},
\it Me, myself and I: a general theory of non-Markovian time-inconsistent stochastic control for sophisticated agents,
\rm Ann. Appl. Probab., {\bf 33} (2023), 1--47.

\bibitem{hernandez-2024} C.~Hern\'{a}ndez and D.~Possama\"{\i},
\it Time-inconsistent contract theory,
\rm Math. Finance,  {\bf34} (2024), 1022--1085.


\bibitem{Hu-2012}  Y. Hu, H. Jin, and X. Y. Zhou,
\it Time-inconsistent stochastic linear-quadratic control,
\rm SIAM J.	Control Optim., {\bf50} (2012), 1548--1572.


\bibitem{Hu-2017}  Y. Hu, H. Jin, and X. Y. Zhou,
\it Time-inconsistent stochastic linear-quadratic control: Characterization and uniqueness of equilibrium,
\rm SIAM J. Control Optim., {\bf55} (2017), 1261--1279.

\bibitem{Lazrak-WY2023} A. Lazrak, H. Wang, and J. Yong,
\it Present-biased lobbyists in linear-quadratic stochastic differential games,
\rm Finance Stoch., {\bf 27} (2023), 947--984.


\bibitem{Lu-2023}  Q. L\"u and B. Ma,
\it Time-inconsistent Stochastic Linear Quadratic Control Problem with Indefinite Control Weight Costs,
\rm Sci. China Math., {\bf67} (2023), 211--236.


\bibitem{Lu-2024}  Q. L\"u and B. Ma,
\it Time-inconsistent linear quadratic optimal control problem for forward-backward stochastic differential equations,
\rm preprint, arXiv:2312.08713.

\bibitem{Mei-Zhu2020} H.~Mei and C.~Zhu,
\it Closed-loop equilibrium for time-inconsistent McKean--Vlasov controlled problem,
\rm SIAM J. Control Optim., {\bf58} (2020), 3842--3867.

\bibitem{Nguwi} J. Y. Nguwi and N. Privault,
\it A constructive approach to existence of equilibria in time-inconsistent stochastic control problems,
\rm SIAM J. Control Optim., {\bf 60} (2022),  674--698.


\bibitem{Pollak}  R. A. Pollak,
\it Consistent planning,
\rm Rev. Econ. Stud., {\bf35} (1968), 185--199.


\bibitem{Strotz-1955}  R. Strotz,
\it Myopia and inconsistency in dynamic utility maximization,
\rm Rev. Econ. Stud., {\bf23} (1955), 165--180.


\bibitem{Sun-Yong-2020}  J. Sun and J. Yong,
\it Stochastic Linear-Quadratic Optimal Control Theory:	Open-Loop and Closed-Loop Solutions,
\rm SpringerBriefs in Math., Springer, Cham, 2020.

\bibitem{Wang-Yong2021} H.~Wang and J.~Yong,
\it Time-inconsistent stochastic optimal control problems and backward stochastic Volterra integral equations,
\rm ESAIM Control Optim. Calc. Var., {\bf27}  (2021),  22.

\bibitem{Wang-Yong-Zhou-2024}  H. Wang, J. Yong, and C. Zhou,
\it Optimal Controls for Forward-Backward Stochastic Differential Equations: Time-Inconsistency and Time-Consistent Solutions,
\rm J. Math. Pures Appl., {\bf190} (2024), 103603.

\bibitem{Wei-Yong-Yu2017} Q.~Wei, J.~Yong, and Z.~Yu,
\it Time-inconsistent recursive stochastic optimal control problems,
\rm SIAM J. Control Optim., {\bf 55} (2017), 4156--4201.

\bibitem{Yong2012} J.~Yong,
\it Time-inconsistent optimal control problems and the equilibrium HJB equation,
\rm Math. Control Relat. Fields, {\bf2} (2012), 271--329.

\bibitem{Yong-2014}  J. Yong,
\it Time-inconsistent optimal control problems,
\rm Proceedings of 2014 ICM, Section 16. Control Theory and Optimization, (2014), 947--969.

\bibitem{Yong-2017}  J. Yong,
\it Linear-quadratic optimal control problems for mean-field stochastic differential equations: time-consistent solutions,
\rm Trans. Amer. Math. Soc.,  {\bf369} (2017), 5467--5523.

\bibitem{Yong-Zhou1999} J.~Yong and X.~Y.~Zhou,
\it Stochastic Control: Hamiltonian Systems and HJB Equations,
\rm Springer-Verlag, New York, 1999.
\end{thebibliography}
\end{document}